\documentclass[final,hidelinks,onefignum,onetabnum]{siamart250211}

\usepackage{enumerate}
\usepackage{bm}
\usepackage{amsmath}
\usepackage{amssymb}
\usepackage{mathrsfs}
\usepackage{graphicx}
\usepackage{cleveref}
\usepackage{todonotes}
\usepackage{makecell}
\ifpdf
\DeclareGraphicsExtensions{.eps,.pdf,.png,.jpg}
\else
\DeclareGraphicsExtensions{.eps}
\fi

\newcommand{\R}{\mathbb{R}}
\newcommand{\Id}{\mathbb{I}}
\let\div\relax \DeclareMathOperator{\div}{div}
\DeclareMathOperator{\tr}{tr}

\DeclareMathOperator{\spanmo}{span}
\newcommand{\norm}[1]{\left\lVert#1\right\rVert}

\newsiamremark{assumption}{Assumption}
\crefname{assumption}{Assumption}{Assumptions}

\headers{High-Order FEM for Multicomponent Flows}
{Aaron Baier-Reinio and Patrick E. Farrell}

\title{High-order finite element methods for three-dimensional multicomponent
convection-diffusion\thanks{\funding{PEF was supported by EPSRC grants 
EP/R029423/1 and EP/W026163/1, and by the Donatio Universitatis
Carolinae Chair ``Mathematical modelling of multicomponent systems''. ABR was supported by a Clarendon scholarship from 
the University of Oxford.}}}

\author{Aaron Baier-Reinio\thanks{Mathematical Institute, University of Oxford, 
	Oxford, OX2 6GG, UK \\
	(\email{aaron.baier-reinio@maths.ox.ac.uk}).
	}
	\and Patrick E.~Farrell\thanks{Mathematical Institute, University of Oxford, Oxford, OX2 6GG, UK and
    Mathematical Institute, Faculty of Mathematics and Physics, Charles University, Czechia
    (\email{patrick.farrell@maths.ox.ac.uk}).
    }
    }

\ifpdf
\hypersetup{
  pdftitle={High-order finite element methods for three-dimensional multicomponent convection-diffusion},
  pdfauthor={Aaron Baier-Reinio and Patrick E.~Farrell}
}
\fi

\begin{document}

\maketitle

\begin{abstract}
We derive and analyze a broad class of finite element methods for numerically
simulating the stationary, low Reynolds number flow of concentrated mixtures of
several distinct chemical species in a common thermodynamic phase.
The underlying partial differential equations that we discretize are the
Stokes--Onsager--Stefan--Maxwell (SOSM) equations, which model bulk momentum
transport and multicomponent diffusion within ideal and non-ideal mixtures.
Unlike previous approaches, the methods are straightforward to implement in two 
and three spatial dimensions, and allow for high-order finite element spaces to be 
employed.
The key idea in deriving the discretization is to suitably reformulate the SOSM
equations in terms of the species mass fluxes and chemical potentials, and
discretize these unknown fields using stable $H(\div)$--$L^2$ finite element pairs.
We prove that the methods are convergent and yield a symmetric linear system for a
Picard linearization of the SOSM equations, which staggers the updates for
concentrations and chemical potentials.
We also discuss how the proposed approach can be extended to the Newton 
linearization of the SOSM equations, which requires the simultaneous solution of
mole fractions, chemical potentials, and other variables.
Our theoretical results are supported by numerical experiments and we present an
example of a physical application involving the microfluidic non-ideal mixing of 
hydrocarbons.
\end{abstract}



\section{Introduction} \label{sec:introduction}

We consider the numerical simulation of
\textit{multicomponent fluids} (or \textit{mixtures}).
Mixtures are fluids comprised of $n\geq 2$ distinct chemical species (or 
\textit{components}) in a common thermodynamic phase (e.g.~liquid or gas).
The defining challenge in modelling mixtures is that the different components are 
not independent; they are instead coupled through physical processes such as 
diffusion or reaction.
Air, for example, is a mixture of oxygen, carbon dioxide, water vapour, nitrogen
and other species, and when modelling airflow in the lungs for heliox therapy,
it is important to track these different components and model their diffusive
interactions \cite{krishna2019diffusing}.

A common but limiting assumption in multicomponent fluid modelling is that
the mixture is in the \textit{dilute regime}.
In this regime the concentration of one component (the \textit{solvent})
dominates that of all others (the \textit{solutes}).
Consequently, the solvent velocity field can be solved for independently of the
solutes, using, for example, the Navier--Stokes equations.
The solute concentrations can then be modelled using decoupled
advection-diffusion equations in which the solvent velocity field is used for 
advection and the diffusive fluxes are modelled using Fick's law
\cite{fick1855ueber}.
However, in \textit{non-dilute} (or \textit{concentrated}) mixtures, this approach
can fail drastically.
In the concentrated regime all components are present in similar amounts, and
Fick's law breaks down as it fails capture the appreciable diffusional forces 
that the components can exert on one another 
\cite{krishna2019diffusing,krishna1997maxwell,wesselingh2000mass}.
This behavior, known as \textit{cross-diffusion}
\cite{cances2023finite,jungel2015boundedness,sun2019entropy},
plays an important
role in many areas of science; examples include
physiology \cite{chang1975some},
combustion \cite{giovangigli2012multicomponent},
electrochemistry \cite{newman2021electrochemical},
geosciences \cite{thorstenson1989gas}
and separation processes \cite{wankat2022separation}.

In this work we derive and analyze finite element methods for capturing two
important physical phenomena in concentrated mixtures: bulk momentum
transport and cross-diffusion.
There is much existing numerical literature that considers these phenomena
separately, but as we will presently discuss, very little numerical literature
studies their coupling.
We consider low Reynolds number flow and model bulk momentum transport
using the Stokes equations; these are well-studied and we refer to
\cite{boffi2013mixed,ern2021finiteII,john2016finite} for expositions on
finite element methods for these equations.
The discretizations we introduce in this work allow for any conforming inf-sup
stable finite element pair for the velocity-pressure formulation of the Stokes
equations to be used.
Examples include the Taylor--Hood \cite{taylor1973numerical}
or Scott--Vogelius \cite{scott1985norm} pairs, on suitable meshes.

We model cross-diffusion using the \textit{Onsager--Stefan--Maxwell}%
\footnote{
Sometimes in the literature called the
\textit{generalized Stefan--Maxwell}
or \textit{Maxwell--Stefan} equations.}%
\textit{(OSM) equations}
\cite{bird2002transport,krishna2019diffusing,
krishna1997maxwell,newman2021electrochemical,newman1965mass,wesselingh2000mass}.
These are named after Onsager, who developed the theory
of irreversible thermodynamics
\cite{onsager1931reciprocal,onsager1931reciprocal2,onsager1945theories},
and Stefan and Maxwell, who independently derived the \textit{Stefan--Maxwell
(SM) equations} \cite{maxwell1866,stefan1871gleichgewicht}
which model cross-diffusion in ideal gaseous mixtures.
It was first realized in \cite{lightfoot1962applicability}
that Onsager's theory and the SM equations are compatible;
their reconciliation yields the more general OSM equations.
The OSM equations are applicable to gases and condensed phases, are valid
in the dilute and concentrated regimes, and can account for a range of
physical phenomena including pressure diffusion, thermal diffusion, thermodynamic
non-idealities and electrochemical effects
\cite{bird2002transport,newman2021electrochemical,wesselingh2000mass}.
The OSM equations also obey the second law of thermodynamics, which states
that entropy generation in a closed system is non-negative.
This ensures that the \textit{Onsager transport matrix}, defined in 
\cref{sec:continuous}, is positive semi-definite 
\cite{onsager1945theories,van2022augmented}, which will be
crucial in our mathematical analysis.

Finite element methods for the OSM equations are proposed and
analyzed in \cite{braukhoff2022entropy,carnes2008local,
jungel2019convergence,mcleod2014mixed,van2022augmented},
but these works assume an isobaric ideal gaseous mixture.
The methods proposed in this work are similar to that of
\cite{mcleod2014mixed}, as we discretize the species mass fluxes in $H(\div)$.
However, the formulation herein is more general than \cite{mcleod2014mixed}, as we
allow for the mixture to be non-ideal, in a gaseous or condensed phase, 
non-isobaric, and we consider the coupling to the Stokes equations.
In fact, the numerical analysis literature that studies the coupling of 
the OSM and Stokes (or Navier--Stokes) equations is sparse.
The only works we know of are
\cite{aznaran2024finite,burman2003bunsen,ern1998thermal,longo2012finite},
but apart from \cite{aznaran2024finite} these papers assume an ideal mixture,
with \cite{burman2003bunsen,ern1998thermal} considering gases and
\cite{aznaran2024finite,longo2012finite} considering both gases
and condensed phases.
Notably, in \cite{aznaran2024finite} the OSM equations are formulated and 
discretized in terms of the species chemical potentials.
This enables \cite{aznaran2024finite} to handle
non-ideal mixtures, and we follow the same approach in this work.

This paper represents a substantial extension of \cite{aznaran2024finite}.
We devise finite element methods for the coupled Stokes and OSM (SOSM) equations,
applicable to non-ideal gases and condensed phases.
To facilitate coupling of the Stokes and OSM equations, the methods in
\cite{aznaran2024finite} require divergence-conforming
symmetric stress elements, i.e.~elements discretizing $H(\div, \mathbb{S})$.
Such elements are difficult to implement
\cite{aznaran2022transformations}, limiting
\cite{aznaran2024finite} to two-dimensional low-order simulations.
We circumvent this challenge and derive a large class of
high-order methods in two and three spatial dimensions.
We do this by reformulating the SOSM equations:
in \cite{aznaran2024finite} the \textit{species velocities} are
discretized in $L^2$, while we instead discretize the
\textit{species mass fluxes} in $H(\div)$.
The $H(\div)$-regularity of the mass fluxes enables coupling of the OSM
and Stokes equations without having to discretize the stress.
This choice of variables also means that the chemical potentials
act as Lagrange multipliers for enforcing the species molar continuity equations,
which leads to the requirement that the discretized mass flux and
chemical potential spaces form stable $H(\div)$--$L^2$ pairs.
Such pairs are standard in finite element software for mixed formulations of the 
Poisson problem; examples include the $\mathbb{BDM}_k$--$\mathbb{DG}_{k-1}$
\cite{brezzi1985two,nedelec1986new}
and $\mathbb{RT}_k$--$\mathbb{DG}_{k-1}$
\cite{raviart1977mixed} pairs.
Another benefit of our reformulation is that we seek the chemical
potentials and pressure in $L^2$, whereas \cite{aznaran2024finite} seeks these
unknowns in a solution-dependent Sobolev space.
This leads to a complicated requirement on the discretization
(see \cite[Assumption 4.2]{aznaran2024finite}) and it is not
clear how to satisfy this requirement when using high-order spaces.
Our formulation circumvents this challenge.

At a theoretical level, we analyze a Picard linearization
of the SOSM equations, which adopts a Gau\ss--Seidel approach by 
segregating the solution of the concentrations from the chemical potentials and 
other variables.
This linearization, which remarkably leads 
to a symmetric system, was introduced in
\cite{aznaran2024finite}, but the variational formulation and discretization
studied here is different.
We establish well-posedness of our formulation of the linearized SOSM 
problem in the continuous and discrete setting.
We also show that, in this linearized setting, our finite element schemes
are quasi-optimal.

The Picard iteration is convenient for analysis, but it can be slow to converge 
to a root and for challenging problems may diverge altogether.
Newton's method \cite{deuflhard2011newton} provides an alternative, but it
requires the solution of larger systems, as the Gau\ss--Seidel staggering is no 
longer possible.
Newton's method can converge to roots quickly and robustly, but it brings 
several major challenges.
Firstly, it is necessary to preserve invariance of the discretization with respect 
to additive constants in the pressure space; doing so requires extra terms in 
the discretization.
We refer to these as density consistency terms
and we find that neglecting them prevents Newton's method from converging.
Secondly, for well-posedness of the stationary nonlinear SOSM problem it is 
necessary to impose constraints on the total number of moles of the species.
Constraints of this form have a clear physical interpretation and can be 
found in the literature on stationary multicomponent flows
(see e.g.~\cite{bulivcek2022existence,giovangigli2015steady}
and also \cite{ern2012mathematical,joubaud2014numerical}).
Each constraint introduces a dense row in the Jacobian matrix; we solve this 
efficiently using the Woodbury formula \cite{hager1989updating} 
on an auxiliary sparse Jacobian.
Together, these adaptations enable the efficient solution of the nonlinear 
SOSM equations with Newton's method for the first time.
A disadvantage of our methods is that, for the nonlinear SOSM problem,  
we empirically observe that the methods converge
sub-optimally in space by one order in the mesh size $h$.
This seems to be due to terms that arise from the coupling
of the Stokes and OSM equations, and it is unclear how to circumvent this 
behavior.

The rest of this paper is organized as follows.
In \cref{sec:continuous} we derive the variational formulation of the
Picard linearized SOSM problem and establish well-posedness at the continuous
level.
In \cref{sec:discretization} we derive the finite element methods and establish
(linearized) discrete well-posedness and quasi-optimality.
In \cref{sec:newton} we outline how the discretizations can be used in conjunction 
with Newton's method to solve the nonlinear SOSM problem.
Numerical experiments are given in \cref{sec:numerical}
and conclusions are drawn in \cref{sec:conclusions}.

\section{Continuous problem and Picard linearization} \label{sec:continuous}

Consider a stationary isothermal mixture of $n \geq 2$ distinct species
in a common thermodynamic phase.
We assume the spatial domain $\Omega \subset \R^d$ ($d \in \{ 2, 3 \}$)
is bounded, connected, and Lipschitz.

\subsection{Notation and governing equations}

Associated to species $i \in \{1, \ldots, n\}$ is its molar concentration
$c_i : \Omega \to \R^{+}$ and velocity $v_i : \Omega \to \R^d$.
The molar mass of species $i$ is denoted by $M_i > 0$ and is a known constant.
The molar flux of species $i$ is $c_i v_i$ and its mass flux is
$M_i c_i v_i$.
We shall primarily work with mass fluxes, which we denote by
$J_i := M_i c_i v_i$.
The stationary molar continuity equation for species $i$ reads
\begin{equation} \label{eq:cty}
	\div (v_i c_i) = r_i
	\quad \textrm{in} \ \Omega,
\end{equation}
with $r_i : \Omega \to \R$ a prescribed reaction term.
The total concentration of the mixture is
$c_T := \textstyle\sum_{j=1}^n c_j$ and the density is
$\rho := \textstyle\sum_{j=1}^n M_j c_j$.
The mass fraction of species $i$ is $\omega_i := M_i c_i / \rho$
and the mass-average (or barycentric) velocity of the mixture is
\begin{equation} \label{eq:massavg}
	v := \textstyle\sum_{j=1}^n \omega_j v_j = \textstyle\sum_{j=1}^n J_j / \rho.
\end{equation}
We refer to \cref{eq:massavg} as the \textit{mass-average constraint}.

We model bulk momentum transport using the steady Stokes momentum equation
\begin{equation} \label{eq:stokes}
	-\div \tau + \nabla p = \rho f
	\quad \textrm{in} \ \Omega,
\end{equation}
where $\tau : \Omega \to \R^{d \times d}_{\textrm{sym}}$ is the
viscous stress, $p : \Omega \to \R$ the pressure and
$f : \Omega \to \R^d$ a prescribed body force.
A derivation of \cref{eq:stokes} in the multicomponent setting is given in
\cite{goyal2017new}.
We assume that $\tau$ and $v$ are related through the Newtonian
constitutive law
\begin{equation} \label{eq:newtonian_const}
	\tau = 2 \eta \epsilon(v)
	+ \big( \zeta - 2 \eta / d \big) \tr(\epsilon(v)) \Id
	:= \mathscr{A}^{-1} \epsilon(v),
\end{equation}
where $\epsilon(v)$ denotes the symmetric gradient of $v$,
$\eta > 0$ is the shear viscosity, $\zeta > 0$ the bulk viscosity,
$\Id$ the $d \times d$ identity and
$\mathscr{A} : \R^{d \times d}_{\textrm{sym}} \to \R^{d \times d}_{\textrm{sym}}$ 
the compliance tensor.
We assume that $\eta$ and $\zeta$ are known constants.

We denote the chemical potential of species $i$ by $\mu_i : \Omega \to \R$.
The isothermal, non-isobaric OSM equations are given by
(see e.g.~\cite{bird2002transport})
\begin{equation} \label{eq:osm_eqns}
	-c_i \nabla \mu_i + \omega_i \nabla p = 
	\textstyle\sum_{j=1}^n \bm{M}_{ij} v_j
	\quad \textrm{in} \ \Omega
	\quad \forall i \in \{1, \ldots, n\}.
\end{equation}
The matrix $\bm{M}$ is called the Onsager transport matrix, and its 
entries are given by
\begin{equation} \label{eq:transport_m_def}
	\bm{M}_{ij} :=
	\begin{cases}
		-\frac{RT c_i c_j}{\mathscr{D}_{ij} c_T}
		& \text{if $i \neq j$}, \\
		\textstyle\sum_{k=1, k \neq i}^n \frac{RT c_i c_k}{\mathscr{D}_{ik} c_T}
		& \text{if $i = j$},
	\end{cases}
\end{equation}
where $R > 0$ is the ideal gas constant, $T > 0$ the temperature and
$\mathscr{D}_{ij}$ the Stefan--Maxwell diffusivities.
Note that $\mathscr{D}_{ii}$ is undefined and
$\mathscr{D}_{ij} = \mathscr{D}_{ji} \ \forall i \neq j$ \cite{bird2002transport}.
It follows that $\bm{M}$ is symmetric with $\sum_{j=1}^n \bm{M}_{ij} = 0 \ \forall 
i$.

Following \cite{aznaran2024finite} we shall assume that each
$\mathscr{D}_{ij}$ is a strictly positive constant,
which implies that $\bm{M}$ is positive-semidefinite, and has exactly one zero
eigenvalue (provided that all species concentrations are strictly positive).
More generally $\mathscr{D}_{ij}$ can be concentration or
pressure dependent \cite{kraaijeveld1993negative}, but studying this case
lies outside the scope of the present work.
However, we expect that under physically reasonable assumptions%
\footnote{%
Assuming that $\bm{M}$ remains positive-semidefinite with exactly one zero
eigenvalue, which was assumed by Onsager \cite{onsager1945theories} through 
consideration of the so-called \textit{dissipation-function}.
All of our Picard linearized analysis in 
\Cref{sec:linearized_wp} and \Cref{sec:discretization}
holds under these more general assumptions on $\bm{M}$ provided that the 
$\mathscr{D}_{ij}$ do not depend on pressure and depend continuously on the 
concentrations.}
on the spectrum of $\bm{M}$,
the numerical schemes proposed here would remain equally applicable in this more 
general setting.

\subsection{Thermodynamic constitutive law} \label{sec:thermo}
Fully describing the mixture requires a
thermodynamic constitutive law that expresses
the concentrations in terms of the pressure and chemical potentials.
We denote this by a map $\mathcal{C}$ which takes $n+1$ functions 
$\mu_1, \ldots, \mu_n, p : \Omega \to \R$
and produces $n$ functions
$c_1, \ldots, c_n : \Omega \to \R^+$.
Hence
\begin{equation} \label{eq:thermo_relation}
	(c_1, \ldots, c_n) = \mathcal{C}(\mu_1, \ldots, \mu_n, p).
\end{equation}
Importantly, the pressure and chemical potentials only appear in the SOSM
equations through their gradients (see \cref{eq:sosm_strong2} below),
and they do not appear in any of our boundary conditions.
Hence these fields are only determined up to additive constants%
\footnote{
These constants are undetermined in the present model,
but in some settings (e.g.~if one is interested in formation energies or 
entropies) they are fully determined and have physical importance.
}.
For $\mathcal{C}$ to be well-defined we require therefore that,
for all constants $C_1, \ldots, C_n, C_p$,
\begin{equation} \label{eq:thermo_wd}
	\mathcal{C}(\mu_1, \ldots, \mu_n, p) =
	\mathcal{C}(\mu_1 + C_1, \ldots, \mu_n + C_n, p + C_p).
\end{equation}

For an ideal gaseous mixture
(see e.g.~\cite{doble2007perry,guggenheim1967thermodynamics})
$\mathcal{C}$ can be defined by means of
\begin{equation} \label{eq:idealgas0}
	c_i = \frac{p^{\ominus}}{RT} 
	\exp \Big( \frac{\mu_i - \mu_i^{\ominus}}{RT} \Big)
	\quad \textrm{in} \ \Omega
	\quad \forall i \in \{ 1, \ldots, n \},
\end{equation}
where $\mu_1^{\ominus}, \ldots, \mu_n^{\ominus}$
are reference chemical potentials and $p^{\ominus}$ a reference pressure.
In general $\mu_i^{\ominus}$ depends on $T$ and $p^{\ominus}$ only.
Both $p^{\ominus}$ and $\mu_i^{\ominus}$ can be eliminated from 
\cref{eq:idealgas0}.
Indeed, since $\mu_i$ is undetermined up to additive constants,
\cref{eq:idealgas0} leaves $c_i$ undetermined up to positive multiplicative 
constants.
We can remove this indeterminacy by enforcing that the total number of moles
of species $i$ is a prescribed value $N_i > 0$,
\begin{equation} \label{eq:total_moles_constraint}
	\textstyle\int_{\Omega} c_i \mathop{\mathrm{d}x} = N_i
	\quad \forall i \in \{ 1, \ldots, n \}.
\end{equation}
With this indeterminacy removed, there is no loss of generality in expressing 
\cref{eq:idealgas0} as
\begin{equation} \label{eq:idealgas1}
	c_i = c^{\ominus} \exp \Big( \frac{\mu_i - \mu_i^{\textrm{aux}}}{RT} \Big)
	\quad \textrm{in} \ \Omega
	\quad \forall i \in \{ 1, \ldots, n \},
\end{equation}
where $\mu_i^{\textrm{aux}} \in \R$ is an auxiliary constant that is chosen so that
\cref{eq:total_moles_constraint} holds, and $c^{\ominus} > 0$ is any fixed constant
with units of concentration%
\footnote{
The actual value of $c^{\ominus}$ is irrelevant, since scaling
$c^{\ominus}$ is equivalent to shifting $\mu_i^{\textrm{aux}}$. We are being
pedantic and include it in \cref{eq:idealgas1} for dimensional consistency,
so that the right-hand side has units of concentration.}.
Defining $\mathcal{C}$ through \cref{eq:total_moles_constraint,eq:idealgas1},
we observe that \cref{eq:thermo_wd} is satisfied.
To motivate \cref{eq:total_moles_constraint}, note that in the transient 
setting, prescription of $\int_{\Omega} c_i \mathop{\mathrm{d}x}$ is unnecessary
since its value is determined by the transient molar continuity equation
$\partial_t c_i + \nabla \cdot (v_i c_i) = r_i$.
Indeed, integration yields
$\frac{\mathrm{d}}{\mathrm{d}t} \int_{\Omega} c_i \mathop{\mathrm{d}x}
= \int_{\Omega} [r_i - \nabla \cdot (v_i c_i)] \mathop{\mathrm{d}x}$,
so that time evolution of $\int_{\Omega} c_i \mathop{\mathrm{d}x}$
is determined by the flux $c_i v_i$, reaction term $r_i$ and initial 
condition on $c_i$.
However, in the steady case the molar continuity equation \cref{eq:cty}
no longer determines the total number of moles of species $i$ and it is 
natural to instead impose \cref{eq:total_moles_constraint} as an additional 
constraint.

In the completely general setting, where the mixture can be in a gaseous or
condensed phase and possibly non-ideal, one must consider the mole
fractions $x_i := c_i / c_T$.
A general constitutive law can be expressed using
$n$ partial molar Gibbs functions $G_i : \R^{n+2} \to \R$
and partial molar volume functions $V_i : \R^{n+2} \to \R$,
whose arguments are the state variables $T, p, x_1, \ldots, x_n$.
These functions are derived from partial derivatives of the Gibbs free energy
of the mixture (see e.g.~\cite{doble2007perry,guggenheim1967thermodynamics}),
and we assume that they are known.
The constitutive law is then an algebraic system of equations:
\begin{subequations} \label{eq:general_const_law}
\begin{alignat}{2}
	\mu_i - \mu_i^{\textrm{aux}} &= 
	G_i(T, p - p^{\textrm{aux}}, x_1, \ldots, x_n)
	&&\quad \textrm{in} \ \Omega
	\quad \forall i \in \{ 1, \ldots, n \},
	\label{eq:general_const_law_1} \\
	1 / c_T &=
	\textstyle\sum_{j=1}^n x_j V_j(T, p - p^{\textrm{aux}}, x_1, \ldots, x_n)
	&&\quad \textrm{in} \ \Omega, \label{eq:general_const_law_2}
\end{alignat}
\end{subequations}
where $\mu_i^{\textrm{aux}}, p^{\textrm{aux}} \in \R$ are auxiliary constants
reflecting the indeterminacy of $\mu_i, p$ up to additive constants.
These auxiliary constants play an analogous role to the $\mu_i^{\textrm{aux}}$
in \cref{eq:idealgas1}, but now in the setting of a general constitutive relation
formulated using mole fractions.
The system in \cref{eq:general_const_law_1} determines
$x_1, \ldots, x_n$, which can then be used in
\cref{eq:general_const_law_2} to determine $c_T$.
The concentrations are then given by $c_i = c_T x_i$.

By the Gibbs--Duhem relation, only $n-1$ of the equations in
\cref{eq:general_const_law_1} are independent
\cite{guggenheim1967thermodynamics}.
Likewise, only $n-1$ of the mole fractions are independent as
$\textstyle\sum_{j=1}^n x_j = 1$.
Therefore, although the system in
\cref{eq:general_const_law}
constitutes $n+1$ equations
in the $n+1$ unknowns $c_T, x_1, \ldots, x_n$,
only $n$ of these equations and $n$ of these unknowns are independent.
In particular, to remove the indeterminacy caused by the $n+1$ auxiliary constants
$\mu_1^{\textrm{aux}}, \ldots, \mu_n^{\textrm{aux}}, p^{\textrm{aux}}$,
it suffices to prescribe at most $n$ constraints.
A reasonable choice of constraints are those in \cref{eq:total_moles_constraint};
however, depending on the form of $G_i$ and $V_i$
it may not always be possible to satisfy these.
For example, if $V_1, \ldots, V_n$ are bounded below by a constant
$V > 0$ (this is the case in \cref{sec:benzene_cyclohexane}),
then multiplying \cref{eq:general_const_law_2} by $c_T$ and integrating
over $\Omega$ reveals that
$\int_{\Omega} 1 \mathop{\mathrm{d}x} \geq
V \textstyle\sum_{j=1}^n \int_{\Omega} c_j \mathop{\mathrm{d}x}$.
Clearly, in this case $\int_{\Omega} c_i \mathop{\mathrm{d}x}$
cannot be prescribed arbitrarily.
A related complication is that sometimes
it is only possible for $k < n$ constraints to be imposed
without overdetermining the system in \cref{eq:general_const_law}
(again, see \cref{sec:benzene_cyclohexane}).
Attempting to analyze exactly what constraints (and how many) can
legitimately be imposed in general is outside the scope of this work.
In what follows, we assume that an appropriate set of constraints have been
suitably chosen so that $\mathcal{C}$ is well-defined and satisfies
\cref{eq:thermo_wd}.

\subsection{Augmentation and change of variables}

An important property of the Onsager transport matrix $\bm{M}$
(recall \cref{eq:transport_m_def}) is that,
assuming the concentrations are strictly positive,
$\bm{M}$ is positive semi-definite
with nullspace spanned by $(1, \ldots, 1)^T$
\cite{onsager1945theories,van2022augmented}.
Positive semi-definiteness of $\bm{M}$ ensures that entropy generation is 
non-negative \cite[eq.~(1.13)]{van2022augmented},
while the nullspace of $\bm{M}$ ensures that bulk convection is
nondissipative \cite[eqs.~(1.15-1.16)]{aznaran2024finite}.
These properties of $\bm{M}$ are not only of physical importance; they have
crucial implications in our numerics, as we now describe.

In our forthcoming variational formulation,
positive semi-definiteness of $\bm{M}$ ensures that a
certain bilinear form is positive semi-definite.
However, for establishing well-posedness
we would prefer this bilinear form to be coercive.
We accomplish this using the augmentation strategy of
\cite[sect.~1.3]{aznaran2024finite}
(see also \cite{ern1994multicomponent,helfand1960inversion,van2022augmented}),
in which multiples of the mass-average constraint \cref{eq:massavg}
are added to the Stokes and OSM equations
\cref{eq:stokes,eq:osm_eqns}.
As in \cite{aznaran2024finite}, we also only strongly enforce the 
divergence of the mass-average constraint,
as this balances the number of equations and unknowns in the SOSM problem.
This leads to the augmented equations \cite[sect.~1.5]{aznaran2024finite}%
\begin{subequations} \label{eq:sosm_previous}
\begin{alignat}{3}
	-c_i \nabla \mu_i + \omega_i \nabla p  + \gamma \omega_i v
	&= \textstyle\sum_{j=1}^n \bm{M}_{ij}^{\gamma} v_j
	&&\quad \ \textrm{in} \ \Omega \quad \forall i,
	\label{eq:sosm_previous_a} \\
	-\div \tau + \nabla p
	+ \gamma v - \gamma \textstyle\sum_{j=1}^n \omega_j v_j
	&= \rho f
	&&\quad \ \textrm{in} \ \Omega,
	\label{eq:sosm_previous_b}  \\
	\div v 
	&= \div \big( \textstyle\sum_{j=1}^n \omega_j v_j \big)
	&&\quad \ \textrm{in} \ \Omega,
	\label{eq:sosm_previous_c}
\end{alignat}
\end{subequations}
where $\gamma > 0$ is a constant user-chosen augmentation parameter, and
\begin{equation} \label{eq:aug_transport_mat}
	\bm{M}^{\gamma}_{ij} := \bm{M}_{ij} + \gamma \omega_i \omega_j
\end{equation}
is an augmented transport matrix, which turns out to be
positive-definite \cite{aznaran2024finite,van2022augmented}.

As discussed in \cref{sec:introduction}, the formulation in
\cite{aznaran2024finite} discretizes the species velocities $v_i$ in $L^2$
which, in turn, requires the stress to be discretized in $H(\div, \mathbb{S})$.
This limits \cite{aznaran2024finite} to two-dimensional low-order simulations.
We circumvent this by reformulating the SOSM equations
in terms of the mass fluxes and we solve for these instead of the
velocities; this will allow us to eliminate the stress altogether.
It will be convenient to introduce the density reciprocal $\Psi := 1 / \rho$.
We divide the augmented OSM equations in \cref{eq:sosm_previous_a} by $M_i c_i$ as
this will lead to our forthcoming linearization to be symmetric;
written in terms of the mass fluxes, the augmented OSM equations are
\begin{equation*}
	-M_i^{-1} \nabla \mu_i + \Psi \nabla p + \gamma \Psi v
	= \textstyle\sum_{j=1}^n \widetilde{\bm{M}}^{\gamma}_{ij} J_j
	\quad \textrm{in} \ \Omega
	\quad \forall i \in \{1, \ldots, n\},
\end{equation*}
where $\widetilde{\bm{M}}^{\gamma}_{ij}$ is a scaling of
the augmented transport matrix in \cref{eq:aug_transport_mat},
\begin{equation} \label{eq:aug_transport_tilde}
	\widetilde{\bm{M}}^{\gamma}_{ij} := 
	\bm{M}^{\gamma}_{ij} / (M_i M_j c_i c_j)
	= \bm{M}_{ij} / (M_i M_j c_i c_j)
	+ \gamma \Psi^2.
\end{equation}
Moreover, rewriting the augmented Stokes momentum equation in
\cref{eq:sosm_previous_b} in terms of the mass fluxes,
and eliminating the stress via \cref{eq:newtonian_const}, we obtain
\begin{equation*}
	-\div \big( \mathscr{A}^{-1} \epsilon(v) \big) + \nabla p + \gamma v - \gamma 
	\Psi 
	\textstyle\sum_{j=1}^n J_j
	= \rho f
	\quad \textrm{in} \ \Omega.
\end{equation*}

\subsection{Full SOSM problem statement} \label{sec:full_problem}

In the following we use bar notation for $n$-tuples; for example
we write $\bar{J} = (J_1, \ldots, J_n)$ and $\bar{\mu} = (\mu_1, \ldots, \mu_n)$.
In this work we consider the following formulation of the SOSM problem.
Given data $f, \bar{r}$, find
a barycentric velocity $v$,
pressure $p$,
mass fluxes $\bar{J}$ and
chemical potentials $\bar{\mu}$
such that
{\allowdisplaybreaks
\begin{subequations} \label{eq:sosm_strong2}
\begin{alignat}{3}
	-M_i^{-1} \nabla \mu_i + \Psi \nabla p + \gamma \Psi v
	&= \textstyle\sum_{j=1}^n \widetilde{\bm{M}}^{\gamma}_{ij} J_j
	&&\quad \ \textrm{in} \ \Omega
	&&\quad \forall i,
	\label{eq:sosm_strong3_a} \\
	-\div \big( \mathscr{A}^{-1} \epsilon(v) \big)
	+ \nabla p + \gamma v - \gamma \Psi 
	\textstyle\sum_{j=1}^n J_j
	&= \rho f
	&&\quad \ \textrm{in} \ \Omega,
	\label{eq:sosm_strong3_b}  \\
	M_i^{-1} \div J_i
	&= r_i
	&&\quad \ \textrm{in} \ \Omega
	&&\quad \forall i,
	\label{eq:sosm_strong3_c} \\
	\div v 
	&= \div \big( \Psi \textstyle\sum_{j=1}^n J_j \big)
	&&\quad \ \textrm{in} \ \Omega,
	\label{eq:sosm_strong3_d}
\end{alignat}
\end{subequations}
}%
subject to the constitutive law in \cref{eq:thermo_relation},
and Dirichlet boundary conditions%
\footnote{
Boundary conditions on the total mass flux $\rho v$
can also be considered (see \cite{aznaran2024finite} and 
\cref{sec:benzene_cyclohexane}).
}%
\begin{equation} \label{eq:sosm_strong2_bcs}
	v = v_D, \quad J_i \cdot n = J_{D,i}
	\quad \textrm{on} \ \partial \Omega
	\quad \forall i \in \{1, \ldots, n\}.
\end{equation}
The Dirichlet data must satisfy compatibility conditions
$M_i^{-1} \int_{\partial \Omega} J_{D,i} \cdot n \mathop{\mathrm{d}s}
= \int_{\Omega} r_i \mathop{\mathrm{d}x}$.

Recall that $\rho$ is the density,
$\Psi = 1 / \rho$ the density reciprocal,
$M_1, \ldots, M_n$ the molar masses,
$\gamma > 0$ the augmentation parameter,
$\mathscr{A}$ the compliance tensor defined through
\cref{eq:newtonian_const} and
$\widetilde{\bm{M}}^{\gamma}_{ij}$
is a modified Onsager transport matrix defined through
\cref{eq:aug_transport_tilde,eq:transport_m_def}.
As in \cite{aznaran2024finite}, 
we make the following regularity assumptions on the data.
\begin{assumption}[Data regularity] \label{asm:data_regularity}
	We assume the data regularity $f, r_1, \ldots, r_n \in L^2(\Omega)$,
	$v_D \in H^{1/2}(\partial \Omega)^d$ and
	$J_{D,i} \in H^{-1/2}(\partial \Omega)$ for all $i \in \{1, \ldots, n\}$.
\end{assumption}
\subsection{Picard linearization}

We first study a Picard linearization of the SOSM problem
which follows a Gau\ss--Seidel approach in which the
concentrations $\bar{c}$ are frozen and treated as fixed fields.
In practice this can be done at every step of a fixed point iteration,
in which one uses the concentrations from the previous step
(computed via the constitutive law in \cref{eq:thermo_relation})
to linearize the problem at the current step.
Fixing the concentrations also allows for
$\rho, \Psi$ and $\widetilde{\bm{M}}^{\gamma}_{ij}$ to be fixed, since they
can be expressed in terms of the concentrations through simple algebraic formulae.
Consequently the SOSM equations in \cref{eq:sosm_strong2} become linear
in the unknown fields $v, p, \bar{J}, \bar{\mu}$.
It is these $2+2n$ fields that we treat as unknowns in our formulation of the
Picard linearization.

We cast the Picard linearization of \cref{eq:sosm_strong2,eq:sosm_strong2_bcs}
as a variational problem.
For the integrals in the variational formulation to be well-defined, as in \cite{aznaran2024finite} we require 
the following physically reasonable assumptions on the density and concentration 
fields.
\begin{assumption}[Concentration and density regularity] 
\label{asm:conc_regularity}
	We assume that
	$\rho \in W^{1,\infty}(\Omega)$,
	and that
	$c_i \in L^{\infty}(\Omega)$ and
	$c_i \geq c_{\textrm{min}}$ a.e.~in $\Omega$ for all $i \in \{1, \ldots, n\}$ where
	$c_{\textrm{min}} > 0$ is a constant.
	Since $\rho = \sum_{j=1}^n M_j c_j$ it also follows that
	$\Psi = 1 / \rho \in W^{1,\infty}(\Omega)$.
\end{assumption}

Let $(\cdot, \cdot)_{\Omega}$ denote the $L^2$-inner-product on $\Omega$
with corresponding norm $\norm{\cdot}_{L^2(\Omega)}$.
We use standard notation (see e.g.~\cite{ern2021finiteI}) for the Sobolev spaces
$H^1(\Omega), H(\div; \Omega)$ and norms
$\norm{\cdot}_{H^1(\Omega)}, \norm{\cdot}_{H(\div; \Omega)}$.
Their counterparts with vanishing traces are denoted by
\begin{equation*}
	H_0^1(\Omega)
	= \{ u \in H^1(\Omega) : u|_{\partial \Omega} = 0 \},
	\quad
	H_0(\div; \Omega)
	= \{ K \in H(\div; \Omega) : K \cdot n|_{\partial \Omega} = 0 \}.
\end{equation*}
We also consider the space
$L_0^2(\Omega) = \{ q \in L^2(\Omega) : (q, 1)_{\Omega} = 0\}$.

Proceeding formally, we multiply \cref{eq:sosm_strong3_a} by a
test function $K_i \in H_0(\div; \Omega)$,
integrate over $\Omega$, and integrate
the terms with gradients by parts, yielding
\begin{equation} \label{eq:sosm_va_a}
	( M_i^{-1} \mu_i, \div K_i)_{\Omega}
	- (p, \div (\Psi K_i))_{\Omega}
	+ (\gamma \Psi v, K_i)_{\Omega}
	= \textstyle\sum_{j=1}^n \big( \widetilde{\bm{M}}^{\gamma}_{ij} J_j, K_i 
	\big)_{\Omega}.
\end{equation}
Next, we multiply \cref{eq:sosm_strong3_b} by a test function 
$u \in H_0^1(\Omega)^d$
and integrate the viscous and pressure terms by parts. This yields
\begin{equation} \label{eq:sosm_va_b}
	(\mathscr{A}^{-1} \epsilon(v), \nabla u)_{\Omega} - (p, \div u)_{\Omega}
	+ \big(\gamma v - \gamma \Psi \textstyle\sum_{j=1}^n J_j, u \big)_{\Omega}
	= (\rho f, u)_{\Omega}.
\end{equation}
Moreover, since $\mathscr{A}$ is defined by
\cref{eq:newtonian_const},
the viscous terms simplify to
\begin{equation} \label{eq:sosm_va_c}
	(\mathscr{A}^{-1} \epsilon(v), \nabla u)_{\Omega} = 2 \eta (\epsilon(v), 
	\epsilon(u))_{\Omega} + \lambda (\div v, \div u)_{\Omega},
\end{equation}
where $\lambda := \zeta - 2\eta / d$ is a Lamé coefficient.
Finally, the variational analogues of
\cref{eq:sosm_strong3_c,eq:sosm_strong3_d} are simply:
for all test functions $w_1, \ldots, w_n, q \in L_0^2(\Omega)$,
\begin{align}
	\textstyle\sum_{i=1}^n (M_i^{-1} \div J_i, w_i)_{\Omega} &= 
	\textstyle\sum_{i=1}^n (r_i, w_i)_{\Omega}, \label{eq:sosm_va_d} \\
	(\div v, q)_{\Omega} &= 
	(\div\big(\Psi \textstyle\sum_{j=1}^n J_j \big), q)_{\Omega}.
	\label{eq:sosm_va_e}
\end{align}

\cref{asm:data_regularity,asm:conc_regularity} imply that the integrals in
\cref{eq:sosm_va_a,eq:sosm_va_b,eq:sosm_va_c,eq:sosm_va_d,eq:sosm_va_e}
are well-defined if $(v, p) \in H^1(\Omega)^d \times L_0^2(\Omega)$
and $(J_i, \mu_i) \in H(\div; \Omega) \times L_0^2(\Omega)$ for all $i$.
These are the spaces we will employ in our variational formulation.
Note that we seek the pressure and chemical potentials in $L_0^2(\Omega)$
as opposed to $L^2(\Omega)$ to address the fact that they are only determined up
to additive constants.
We also introduce the product spaces
\begin{equation} \label{eq:sosm_va_spaces}
	V = H^1(\Omega)^d \times H(\div; \Omega)^n, \
	V_0 = H_0^1(\Omega)^d \times H_0(\div; \Omega)^n, \
	Q = L_0^2(\Omega) \times L_0^2(\Omega)^n.
\end{equation}
Note that $V_0 \subset V$ consists of the traceless functions in $V$.
Consider the bilinear forms%
\begin{subequations}
	\begin{align}
		&\begin{cases} \label{eq:a_v_cts_defn}
			\begin{split}
				&a_v : H^1(\Omega)^d \times H_0^1(\Omega)^d \to \R, \\
				&a_v(v, u) :=
				2 \eta (\epsilon(v), \epsilon(u))_{\Omega}
				+ \lambda (\div v, \div u)_{\Omega},
			\end{split}
		\end{cases} \\
		&\begin{cases} \label{eq:a_o_cts_defn}
			\begin{split}
				&a_o : V \times V_0 \to \R, \\
				&a_o((v, \bar{J}), (u, \bar{K})) :=
				\begin{aligned}[t]
					&\gamma \big( v - \Psi \textstyle\sum_{j=1}^n J_j,
					u - \Psi \textstyle\sum_{i=1}^n K_i \big)_{\Omega} \\
					&+ \textstyle\sum_{i,j=1}^n
					\big( \widetilde{\bm{M}}_{ij} J_j, K_i \big)_{\Omega},
				\end{aligned} \\
				&\textrm{ where }
				\widetilde{\bm{M}}_{ij} := \bm{M}_{ij} / (M_i M_j c_i c_j)
				\textrm{ and $\bm{M}_{ij}$ is defined in 
				\cref{eq:transport_m_def}},
			\end{split}
		\end{cases} \\
		&\begin{cases} \label{eq:a_cts_defn}
			\begin{split}
				&a : V \times V_0 \to \R, \\
				&a((v, \bar{J}), (u, \bar{K})) := a_v(v, u)
				+ a_o((v, \bar{J}), (u, \bar{K})),		
			\end{split}
		\end{cases} \\
		&\begin{cases} \label{eq:b_cts_defn}
			\begin{split}
				&b : V \times Q \to \R, \\
				&b((u, \bar{K}), (p, \bar{\mu})) :=
				\begin{aligned}[t]
					&-(p, \div u)_{\Omega}
					+ \textstyle{\sum_{i=1}^n} (p, \div(\Psi K_i))_{\Omega} \\
					&- \textstyle{\sum_{i=1}^n}
					(M_i^{-1} \mu_i, \div K_i)_{\Omega}.
				\end{aligned}
			\end{split}
		\end{cases}
	\end{align}
\end{subequations}

The conditions in
\cref{eq:sosm_va_a,eq:sosm_va_b,eq:sosm_va_c,eq:sosm_va_d,eq:sosm_va_e} are
equivalent to the following problem:
find $((v, \bar{J}), (p, \bar{\mu})) \in V \times Q$
such that $(v, \bar{J})$ satisfy the boundary conditions in
\cref{eq:sosm_strong2_bcs} and
\begin{subequations} \label{eq:sosm_weak}
	\begin{alignat}{2}
		a((v, \bar{J}), (u, \bar{K})) + b((u, \bar{K}), (p, \bar{\mu}))
		&= (\rho f, u)_{\Omega}
		\quad &&\forall (u, \bar{K}) \in V_0, \\
		b((v, \bar{J}), (q, \bar{w})) &= 
		-\textstyle\sum_{i=1}^n (r_i, w_i)_{\Omega} 
		\quad &&\forall (q, \bar{w}) \in Q.
	\end{alignat}
\end{subequations}
This constitutes our variational formulation of the Picard linearized SOSM problem.
Note that $a$ is symmetric on $V_0 \times V_0$
and therefore \cref{eq:sosm_weak} is a symmetric problem.

\subsection{Linearized well-posedness} \label{sec:linearized_wp}

We now establish well-posedness of the problem
in \cref{eq:sosm_weak}.
Without loss of generality we may consider the case of homogeneous
Dirichlet boundary data, so that $(v, \bar{J})$ is sought for in $V_0$.
By linearity the inhomogeneous case can
be recovered by decomposing the solution as
$(v, \bar{J}) = (v_0, \bar{J}_0) + (v_l, \bar{J}_l)$
where $(v_0, \bar{J}_0) \in V_0$ are functions to be sought for and
$(v_l, \bar{J}_l) \in V$ are fixed liftings which satisfy
the inhomogeneous Dirichlet boundary conditions in \cref{eq:sosm_strong2_bcs}.

We equip the product spaces $V$ and $Q$ with their Hilbertian product norms
\begin{align*}
	\norm{(u, \bar{K})}_V^2 &= \norm{u}_{H^1(\Omega)^d}^2
	+ \textstyle\sum_{j=1}^n \norm{K_j}_{H(\div; \Omega)}^2, \\
	\norm{(q, \bar{w})}_Q^2 &= \norm{q}_{L^2(\Omega)}^2
	+ \textstyle\sum_{j=1}^n \norm{w_j}_{L^2(\Omega)}^2.
\end{align*}
One can verify that $a$ and $b$ are bounded in these norms, with constants of
boundedness dependent on the concentrations and density.
The linear forms on the right-hand side of \cref{eq:sosm_weak} are also bounded in 
these norms.
Well-posedness of the problem in \cref{eq:sosm_weak}
follows from the following two standard conditions on $a$ and $b$ (see 
e.g.~\cite[Thm.~4.2.1]{boffi2013mixed}).
\begin{enumerate}[(i)]
	\item An inf-sup condition on $b$:
	For some constant $\beta > 0$, there holds
	\begin{equation} \label{eq:b_infsup_cts}
		\sup_{(u, \bar{K}) \in V_0} 
		\frac{b((u, \bar{K}), (q, \bar{w}))}
		{\norm{(u, \bar{K})}_V } \geq \beta
		\norm{(q, \bar{w})}_Q
		\quad \forall (q, \bar{w}) \in Q.
	\end{equation}
	\item Coercivity of $a$ on $\ker b$:
	For some constant $\alpha > 0$, there holds
	\begin{equation} \label{eq:a_coercive_cts}
		a((u, \bar{K}), (u, \bar{K})) \geq \alpha \norm{(u, \bar{K})}_V^2
		\quad \forall (u, \bar{K}) \in W,
	\end{equation}
	where $W \subset V_0$ is the kernel of $b$, i.e.
	\begin{equation} \label{eq:ker_b_defn_cts}
		W := 	
		\Big\{ (u, \bar{K}) \in V_0 : 
		b((u, \bar{K}), (q, \bar{w})) = 0 \ \forall (q, \bar{w}) \in Q
		\Big\}.
	\end{equation}
\end{enumerate}
\begin{proof}[Proof of the inf-sup condition in \cref{eq:b_infsup_cts}]
Let $Q^*$ denote the dual space of $Q$.
Let $B : V_0 \to Q^*$ be the linear operator corresponding to $b$, i.e.
\begin{equation} \label{eq:B_cts_defn}
	[B (u, \bar{K})] (q, \bar{w})
	= b((u, \bar{K}), (q, \bar{w}))
	\quad \forall ((u, \bar{K}), (q, \bar{w})) \in V_0 \times Q.
\end{equation}
The inf-sup condition in \cref{eq:b_infsup_cts} is
equivalent to the statement that $B$ is a surjection 
(see e.g.~\cite[sect.~4.2.2]{boffi2013mixed}).
To verify surjectivity of $B$, let $l \in Q^*$ be given.
Since $Q= L_0^2(\Omega) \times L_0^2(\Omega)^n$, by the Riesz representation 
theorem there exists functions $l_0, l_1, \ldots, l_n \in L_0^2(\Omega)$ such that
$l (q, \bar{w}) = (l_0, q)_{\Omega}+ \textstyle\sum_{i=1}^n (l_i, w_i)_{\Omega}$
for all $(q, \bar{w}) \in Q$.
Since $\div : H_0(\div; \Omega) \to L_0^2(\Omega)$ is surjective
\cite{girault1986finite},
we can choose $K_i \in H_0(\div; \Omega)$ with $-M_i^{-1} \div K_i = l_i$
for each $i \in \{1, \ldots, n\}$.
Likewise, since $\div : H_0^1(\Omega)^d \to L_0^2(\Omega)$ is surjective
\cite{girault1986finite},
we can choose $u \in H_0^1(\Omega)^d$ with
$\div u = -l_0 + \sum_{i=1}^n \div(\Psi K_i)$.
The definitions of $B$ and $b$ (see \cref{eq:B_cts_defn,eq:b_cts_defn})
reveal that $B(u, \bar{K}) = l$.
\end{proof}

\begin{proof}[Proof of the coercivity condition in \cref{eq:a_coercive_cts}]
	Using the definitions of $W$ and $b$
	(see \cref{eq:ker_b_defn_cts,eq:b_cts_defn})
	one can check that if $(u, \bar{K}) \in W$
	then $\div K_i = 0$ for all $i$.
	The inequality in \cref{eq:a_coercive_cts} will therefore follow
	if we can prove that
	\begin{equation} \label{eq:a_coerciv_conc4}
		a((u, \bar{K}), (u, \bar{K})) \geq \alpha 
		\Big[\norm{u}_{H^1(\Omega)^d}^2
		+ \textstyle\sum_{i=1}^n \norm{K_i}_{L^2(\Omega)^d}^2 \Big]
		\quad \forall (u, \bar{K}) \in V_0.
	\end{equation}
	We emphasize that \cref{eq:a_coerciv_conc4} holds for all
	$(u, \bar{K}) \in V_0$ and not just $(u, \bar{K}) \in W$.
	This fact will be important later on for establishing discrete well-posedness.
	
	We now establish \cref{eq:a_coerciv_conc4}.
	Recalling the definition of $a$ in \cref{eq:a_cts_defn}, we have
	\begin{equation*}
		a((u, \bar{K}), (u, \bar{K}))
		= a_v(u, u) 
		+ a_o((u, \bar{K}), (u, \bar{K})).
	\end{equation*}
	A lower bound for $a_v(u, u)$ (recall \cref{eq:a_v_cts_defn}) can be 
	obtained with standard arguments involving a Korn inequality
	(see e.g.~\cite[Chapter 42]{ern2021finiteII}),
	which yield
	\begin{equation} \label{eq:a_v_bound}
		a_v(u, u) \geq C_1 \delta
		\norm{u}_{H^1(\Omega)^d}^2 \quad \forall u \in H_0^1(\Omega)^d,
	\end{equation}
	for some constant $C_1 > 0$ and $\delta = \min\{ \eta, \zeta \}$.
	To bound $a_o((u, \bar{K}), (u, \bar{K}))$, let $(u, \bar{K}) \in V_0$ and set
	$u_i = K_i / (M_i c_i) \in L^2(\Omega)$.
	The definition of $a_o$ (recall \cref{eq:a_o_cts_defn}) reveals that
	\begin{align} \label{eq:a_osm_bound1}
		\begin{split}
			a_o((u, \bar{K}), (u, \bar{K}))
			=& \textstyle\sum_{i,j=1}^n 
			\big( \widetilde{\bm{M}}_{ij} K_j, K_i \big)_{\Omega} \\
			&+ \gamma
			\big(u - \Psi \textstyle\sum_{j=1}^n K_j, 
			u - \Psi \textstyle\sum_{i=1}^n K_i \big)_{\Omega} \\
			=& \textstyle\sum_{i,j=1}^n
			\big( \bm{M}_{ij} u_j, u_i \big)_{\Omega} \\
			&+ \textstyle\sum_{i,j=1}^n
			\gamma \big( \omega_j(u - u_j), \omega_i (u - u_i) \big)_{\Omega}.
		\end{split}
	\end{align}
	Also, since $\bm{M}_{ij}$ is symmetric with
	$\sum_{j=1}^n \bm{M}_{ij} = 0$ for all $i$, we have
	\begin{equation} \label{eq:a_osm_bound2}
		\textstyle\sum_{i,j=1}^n \big( \bm{M}_{ij} u_j, u_i \big)_{\Omega}
		= \textstyle\sum_{i,j=1}^n 
		\big( \bm{M}_{ij} (u - u_j), u - u_i \big)_{\Omega}.
	\end{equation}
	But
	$\bm{M}_{ij}^{\gamma} = \bm{M}_{ij} + \gamma \omega_i \omega_j$
	(recall \cref{eq:aug_transport_mat}) and hence
	\crefrange{eq:a_osm_bound1}{eq:a_osm_bound2} yield
	\begin{equation} \label{eq:a_osm_bound3}
		a_o((u, \bar{K}), (u, \bar{K})) = 
		\textstyle\sum_{i,j=1}^n \big( \bm{M}^{\gamma}_{ij} (u - u_j), u - u_i
		\big)_{\Omega}.
	\end{equation}
	As argued in the proof of \cite[Lem.~3.3]{aznaran2024finite}, 
	since $\gamma > 0$ the matrix
	$\bm{M}^{\gamma}_{ij}$ is uniformly positive-definite on $\Omega$.
	Hence, by \cref{eq:a_osm_bound3},
	for some constant $C_2 > 0$
	there holds
	\begin{equation} \label{eq:a_osm_bound4}
		a_o((u, \bar{K}), (u, \bar{K}))
		\geq C_2
		\textstyle\sum_{i=1}^n \norm{u - u_i}_{L^2(\Omega)^d}^2.
	\end{equation}
	Note that $C_2$ depends on the concentrations, as discussed in
	\cite[Rem.~4.4]{van2022augmented}.
	
	Finally, let
	$C_3 = 1 / \max \big\{ \norm{M_j c_j}_{L^{\infty}(\Omega)} \big\}_{j=1}^n > 0$.
	Since $u_i = K_i / (M_i c_i)$ we then have
	$\norm{u_i}_{L^2(\Omega)^d} \geq
	C_3 \norm{K_i}_{L^2(\Omega)^d}$ for all $i$.
	Combining \cref{eq:a_v_bound,eq:a_osm_bound4} we find that
	\begin{align*}
		a((u, \bar{K}), (u, \bar{K})) \geq&
		\min \{C_1 \delta / 2, C_2 \}
		\Big[\norm{u}_{H^1(\Omega)^d}^2 + \norm{u}_{L^2(\Omega)^d}^2
		+ \textstyle\sum_{i=1}^n \norm{u - u_i}_{L^2(\Omega)^d}^2
		\Big] \\
		\geq&
		\min \{C_1 \delta / 2, C_2 \}
		\Big[\norm{u}_{H^1(\Omega)^d}^2
		+ (n+1)^{-1} \textstyle\sum_{i=1}^n \norm{u_i}_{L^2(\Omega)^d}^2
		\Big] \\
		\geq&
		\min \{C_1 \delta / 2, C_2 \} \cdot \min \{1, C_3 / (n+1) \} \\
		&\cdot \Big[\norm{u}_{H^1(\Omega)^d}^2
		+ \textstyle\sum_{i=1}^n \norm{K_i}_{L^2(\Omega)^d}^2
		\Big],
	\end{align*}
	where we used the reverse triangle inequality.
	This establishes the bound in \cref{eq:a_coerciv_conc4}.
\end{proof}

\section{Discretization of the Picard linearization} \label{sec:discretization}

We now derive quasi-optimal finite element methods for 
solving the Picard linearized SOSM problem \cref{eq:sosm_weak}.
Let $V^v_h \subset H^1(\Omega)^d$ denote the discrete barycentric
velocity space, $V^J_h \subset H(\div; \Omega)$
the discrete species mass flux space, $P_h \subset L_0^2(\Omega)$
the discrete pressure space and $U_h \subset L_0^2(\Omega)$
the discrete species chemical potential space.
Here $h \in (0, \infty)$ is a parameter generating these spaces
(e.g.~the maximum cell diameter in a mesh of $\Omega$).
We also set
\begin{equation} \label{eq:discrete_spaces_lin}
	V_h := V^v_h \times (V^J_h)^n \subset V
	\quad \textrm{and} \quad
	Q_h := P_h \times (U_h)^n \subset Q.
\end{equation}
The subspaces with vanishing traces are denoted by
$V^v_{0h} := V^v_h \cap H_0^1(\Omega)^d$,
$V^J_{0h} := V^J_h \cap H_0(\div; \Omega)$
and $V_{0h} := V^v_{0h} \times (V^J_{0h})^n \subset V_0$.
For simplicity we take $V^J_h$ and $U_h$ to be independent
of the species index $i \in \{1, \ldots, n\}$, i.e.~for each species we
use the same mass flux and chemical potential space. 
However, the results here can be extended to the case where
different mass flux and chemical potential spaces are used for different species.

A conforming discretization of \cref{eq:sosm_weak} is obtained 
by employing Galerkin's method with $V_{0h} \times Q_h$ as the 
discrete subspace of $V_0 \times Q$.
We assume homogeneous Dirichlet boundary conditions in
\cref{eq:sosm_strong2_bcs}; the inhomogeneous case is addressed
using discrete lifting functions.
The discrete analogue of \cref{eq:sosm_weak} is thus:
find $((v_h, \bar{J}_h), (p_h, \bar{\mu}_h)) \in V_{0h} \times Q_h$ such that
for all $((u_h, \bar{K}_h), (q_h, \bar{w}_h)) \in V_{0h} \times Q_h$ there holds
\begin{subequations} \label{eq:discrete_problem}
	\begin{align}
		a((v_h, \bar{J}_h), (u_h, \bar{K}_h)) 
		+ b((u_h, \bar{K}_h), (p_h, \bar{\mu}_h))
		&= (\rho f, u_h)_{\Omega}, \\
		b((v_h, \bar{J}_h), (q_h, \bar{w}_h)) 
		&= -\textstyle\sum_{i=1}^n (w_{h,i}, r_i)_{\Omega}.
	\end{align}
\end{subequations}
Since \cref{eq:sosm_weak} is a saddle-point problem, well-posedness at the
continuous level is not automatically inherited at the discrete level.
The following theorem
(see e.g.~\cite[Chap.~5]{boffi2013mixed}) outlines conditions that ensure 
discrete well-posedness and quasi-optimality.
\begin{theorem}[Discrete well-posedness] \label{thm:discrete_wp}
	Assume that there exists constants $\tilde{\beta}, \tilde{\alpha} > 0$
	independent of $h$ such that we have (i) the discrete inf-sup condition
	\begin{equation} \label{eq:b_infsup_discrete}
		\sup_{(u_h, \bar{K}_h) \in V_{0h}} 
		\frac{b((u_h, \bar{K}_h), (q_h, \bar{w}_h))}
		{\norm{(u_h, \bar{K}_h)}_V} \geq 
		\tilde{\beta} \norm{(q_h, \bar{w}_h)}_Q
		\quad \forall (q_h, \bar{w}_h) \in Q_h,
	\end{equation}
	and (ii) on the discrete kernel of $b$
	\begin{equation} \label{eq:discrete_kernel}
	W_h := \big\{ 
	(u_h, \bar{K}_h) \in V_{0h} : 
	b((u_h, \bar{K}_h), (q_h, \bar{w}_h)) = 0
	\ \forall (q_h, \bar{w}_h) \in Q_h
	\big\},
	\end{equation}
	there holds the coercivity condition
	\begin{equation} \label{eq:b_coercivity_discrete}
		a((u_h, \bar{K}_h), (u_h, \bar{K}_h)) 
		\geq \tilde{\alpha} \norm{(u_h, \bar{K}_h)}_V^2
		\quad \forall (u_h, \bar{K}_h) \in W_h.
	\end{equation}
	Then the discrete problem in \cref{eq:discrete_problem} is well-posed.
	Also, let $((v, \bar{J}), (p, \bar{\mu})) \in V_0 \times Q$ denote the
	solution of \cref{eq:sosm_weak} and 
	$((v_h, \bar{J}_h), (p_h, \bar{\mu}_h)) \in V_{0h} \times Q_h$
	the solution of \cref{eq:discrete_problem}.
	Then we have the quasi-optimal a priori error estimate
	\begin{equation} \label{eq:discrete_quasioptimal}
		\norm{(v - v_h, \bar{J} - \bar{J}_h)}_V
		+ \norm{(p - p_h, \bar{\mu} - \bar{\mu}_h)}_Q
		\leq C \cdot \mathcal{E}_h,
	\end{equation}
	with $C > 0$ independent of $h$, and 
	$\mathcal{E}_h$ the best approximation error, i.e.
	\begin{equation} \label{eq:best_approx_error}
		\mathcal{E}_h := 
		\inf_{(u_h, \bar{K}_h) \in V_{0h}}
		\norm{(v - u_h, \bar{J} - \bar{K}_h)}_V
		+ \inf_{(q_h, \bar{w}_h) \in Q_h}
		\norm{(p - q_h, \bar{\mu} - \bar{W}_h)}_Q.
	\end{equation}
\end{theorem}

The next two lemmas give conditions
which ensure that \cref{eq:b_infsup_discrete,eq:b_coercivity_discrete}
hold.

\begin{lemma}[Inf-sup stability] \label{lemma:discrete_infsup}
	Condition \cref{eq:b_infsup_discrete} will
	hold provided that,
	for some constants $\beta_1, \beta_2 > 0$ independent of $h$,
	$(V^v_h, P_h)$ and $(V^J_h, U_h)$ satisfy inf-sup conditions%
	\begin{subequations}
	\begin{align}
		\sup_{u_h \in V^v_{0h}}
		\frac{(\div u_h, q_h)_{\Omega}}
		{\norm{u_h}_{H^1(\Omega)^d}} &\geq \beta_1 \norm{q_h}_{L^2(\Omega)}
		\quad \forall q_h \in P_h, \label{eq:discrete_infsup_h1} \\
		\sup_{K_h \in V^J_{0h}}
		\frac{(\div K_h, w_h)_{\Omega}}
		{\norm{K_h}_{H(\div; \Omega)}} &\geq 
		\beta_2 \norm{w_h}_{L^2(\Omega)}
		\quad \forall w_h \in U_h. \label{eq:discrete_infsup_hdiv}
	\end{align}
	\end{subequations}
\end{lemma}
\begin{proof}
The proof is analogous to that of the continuous
inf-sup condition \cref{eq:b_infsup_cts},
but with the need to use $\beta_1$ and $\beta_2$
to obtain an $h$-independent lower bound for $\tilde{\beta}$.
\end{proof}
\begin{lemma}[Discrete coercivity] \label{lemma:discrete_coercivity}
	Condition \cref{eq:b_coercivity_discrete}
	will hold provided there exists a constant $c > 0$ independent of $h$ such that
	\begin{equation} \label{eq:discrete_coercivity_inverse}
		\norm{K_h}_{L^2(\Omega)^d} \geq c
		\norm{K_h}_{H(\div; \Omega)}
	\end{equation}
	for all $K_h \in V^J_{0h}$ satisfying
	$(\div K_h, w_h)_{\Omega} = 0 \ \forall w_h \in U_h$.
	Note that the condition in \cref{eq:discrete_coercivity_inverse} holds
	with $c=1$ for divergence-free pairs, i.e.~pairs such that
	if $K_h \in V^J_{0h}$ satisfies 
	$(\div K_h, w_h)_{\Omega} = 0 \ \forall w_h \in U_h$
	then $\div K_h = 0$.
\end{lemma}
\begin{proof}
	Using the definition of $W_h$ and $b$
	(see \cref{eq:b_cts_defn,eq:discrete_kernel})
	one verifies that for all
	$(u_h, \bar{K}_h) \in W_h$,
	there holds
	$(\div K_{h,i}, w_h)_{\Omega} = 0 \ \forall w_h \in U_h \ \forall i$.
	The bounds in \cref{eq:a_coerciv_conc4} and
	\cref{eq:discrete_coercivity_inverse} then
	immediately yield that \cref{eq:b_coercivity_discrete} holds with
	$\tilde{\alpha} = \alpha \cdot \min \{1, c^2 \}$.
\end{proof}

Many standard finite element spaces satisfy the hypotheses of
\cref{lemma:discrete_infsup,lemma:discrete_coercivity}.
The barycentric velocity and pressure pair $(V^v_h, P_h)$ needs
to satisfy the inf-sup condition in \cref{eq:discrete_infsup_h1},
which can be accomplished by employing any conforming inf-sup stable
Stokes pair.
Examples include the Taylor--Hood \cite{taylor1973numerical} or
Scott--Vogelius \cite{scott1985norm} pairs.
The mass flux and chemical potential pair
$(V^J_h, U_h)$ must satisfy the conditions in
\cref{eq:discrete_coercivity_inverse,eq:discrete_infsup_hdiv},
which can be accomplished by employing, for example,
$\mathbb{BDM}_k$--$\mathbb{DG}_{k-1}$
\cite{brezzi1985two,nedelec1986new}
or $\mathbb{RT}_k$--$\mathbb{DG}_{k-1}$
\cite{raviart1977mixed} pairs
(these pairs are divergence-free).
Optimal, high-order spatial convergence rates can be achieved when
using the Taylor--Hood pair with degree $k \geq 2$~\cite{boffi1997three}
or Scott--Vogelius on suitable meshes (e.g.~$k \geq d$ on barycentrically-refined meshes~\cite{zhang2004new})
for $(V^v_h, P_h)$,
and $\mathbb{BDM}_k$--$\mathbb{DG}_{k-1}$ or
$\mathbb{RT}_k$--$\mathbb{DG}_{k-1}$
pairs for $(V^J_h, U_h)$.
For sufficiently smooth solutions,
the bound in \cref{eq:discrete_quasioptimal} then
predicts an optimal rate of
\begin{equation} \label{eq:optimal_rate}
	\norm{(v - v_h, \bar{J} - \bar{J}_h)}_V
	+ \norm{(p - p_h, \bar{\mu} - \bar{\mu}_h)}_Q
	= \mathcal{O}(h^k),
\end{equation}
where $h$ is the maximum cell diameter of
a shape-regular triangulation of $\Omega$.
We emphasize, however, that we have only proven these spatial rates
of convergence for the Picard linearization of the SOSM problem.

In single-component incompressible flow it can be advantageous to 
use divergence-free Stokes elements, as these yield velocity approximations that 
are independent of the pressure (i.e.~\textit{pressure-robust}) 
\cite{john2017divergence}.
For this reason Scott--Vogelius may be preferable to Taylor--Hood in the 
single-component incompressible setting, as the former is divergence-free 
whereas the latter is not.
However, in the multicomponent setting there does not seem to be
an analogue of the pressure-robustness phenomenon, since the mass-average 
constraint couples the barycentric velocity to the density and mass fluxes,
which in turn are coupled to the pressure through the OSM equations.
Hence, we do not expect a pressure-robust velocity approximation even with
divergence-free Stokes elements.
For this reason we prefer to use Taylor--Hood in our simulations,
as it has no apparent disadvantage relative to Scott--Vogelius but does not require special 
meshes.
It is less clear whether to use $\mathbb{BDM}_k$ or $\mathbb{RT}_k$ for the 
mass fluxes, since both lead to the same convergence rates in 
\cref{eq:optimal_rate}.
Note that the $\mathbb{BDM}_k$ finite element space includes all polynomials of 
degree $k$ whereas the $\mathbb{RT}_k$ space is a strict subspace of these
\cite{boffi2013mixed}.
Hence, for a given mesh and polynomial degree $k$, a $\mathbb{BDM}_k$ 
discretization may be sightly more accurate but also slightly more expensive
than that of $\mathbb{RT}_k$.
In practice the value of this (fairly minor) trade-off will be problem- and 
user-dependent.

\section{Nonlinear monolithic discretization and Newton's method} 
\label{sec:newton}

When combined with fixed point iteration, the Picard linearization
can be used to solve the nonlinear SOSM problem.
However, for challenging problems this iteration often does not converge to a 
root, or does so at a prohibitively slow rate.
A more robust approach is to use Newton's method \cite{deuflhard2011newton};
this requires a monolithic discretization of the nonlinear SOSM problem,
which we now outline.

In the monolithic setting, the Gau\ss--Seidel staggering employed in the Picard
linearization is no longer possible.
We still seek $((v_h, \bar{J}_h), (p_h, \bar{\mu}_h))$
in $V_h \times Q_h$ (recall \cref{eq:discrete_spaces_lin}),
but we must introduce additional unknowns to discretize the constitutive law 
in \cref{eq:thermo_relation}.
It is natural to seek the mole fractions in a
finite element space $X_h \subset L^2(\Omega)$
and enforce the $L^2$-projection of \cref{eq:general_const_law_1} into $X_h$.
Hence the discrete analogue of \cref{eq:general_const_law_1} is to
find $\bar{x}_h \in (X_h)^n$
such that for all $\bar{y}_h \in (X_h)^n$,
\begin{equation} \label{eq:general_const_law_discrete_1}
	(\mu_{h,i} - \mu_i^{\textrm{aux}}, y_{h,i})_{\Omega}
	= (G_i(T, p_h - p^{\textrm{aux}}, x_{h,1}, \ldots, x_{h,n}), y_{h,i})_{\Omega}
	\quad \forall i \in \{ 1, \ldots, n \}.
\end{equation}

As discussed in \cref{sec:thermo}, the constants
$\mu_1^{\textrm{aux}}, \ldots, \mu_n^{\textrm{aux}}, p^{\textrm{aux}} \in \R$
reflect the indeterminacy of the pressure and chemical potentials
up to additive constants.
To avoid solving for these constants we use the following trick.
Recall that in \cref{sec:discretization} we employed discrete chemical
potential and pressure spaces $P_h, U_h \subset L_0^2(\Omega)$.
Using $L_0^2(\Omega)$ instead of $L^2(\Omega)$ is convenient for the
analysis, but at an implementation level
it is easier to work with the analogues of these spaces
containing all constant functions,
\begin{equation} \label{eq:discrete_cp_p_spaces_constants}
	P_h' := P_h \oplus \spanmo\{1\}, \
	U_h' := U_h \oplus \spanmo\{1\}
	\textrm{ and }
	Q_h' := P_h' \times (U_h')^n.
\end{equation}
If we seek $(p_h, \bar{\mu}_h) \in Q_h'$
then there is no need for
$\mu_1^{\textrm{aux}}, \ldots, \mu_n^{\textrm{aux}}, p^{\textrm{aux}}$
to appear in \cref{eq:general_const_law_discrete_1}
as they can be absorbed into $(p_h, \bar{\mu}_h)$.
This is a simple trick, but using $Q_h'$ instead
of $Q_h$ requires the introduction of \textit{density consistency} terms,
as discussed further below.

Satisfaction of \cref{eq:general_const_law_discrete_1}
does not ensure that the discrete mole fractions exactly sum to unity,
i.e.~$\sum_{j=1}^n x_{h,j} = 1$ will only hold approximately.
One can instead solve for the $n-1$ discrete mole fractions
$x_{h,2}, \ldots, x_{h,n}$ and express $x_{h,1} = 1 - \sum_{j=2}^n x_{h,j}$.
However, this breaks permutational symmetry of the species
and it does not ensure that all $n$ equations in 
\cref{eq:general_const_law_1}
hold in an approximate sense 
(e.g.~as in \cref{eq:general_const_law_discrete_1}).
We therefore prefer to enforce \cref{eq:general_const_law_discrete_1} and
in the computations we verify that $\sum_{j=1}^n x_{h,j} = 1$
holds to a satisfactory extent.
Of course, we can still introduce normalized discrete mole fractions
$x_{h,i}^{\textrm{nm}} := x_{h,i} / \textstyle\sum_{j=1}^n x_{h,j}$.
Using \cref{eq:general_const_law_2}
and since $c_i = c_T x_i$, a discrete expression for the concentrations
(as a function of $p_h, \bar{x}_h$ only) is
\begin{equation} \label{eq:discrete_conc}
	c_{h,i} := 
	\Big[ \textstyle\sum_{j=1}^n 
	x_{h,j}^{\textrm{nm}}
	V_j(T, p_h, x_{h,1}^{\textrm{nm}}, \ldots, 
	x_{h,n}^{\textrm{nm}})
	\Big]^{-1} x_{h,i}^{\textrm{nm}}
	\quad \forall i \in \{ 1, \ldots, n \}.
\end{equation}

Analogously to \cref{eq:b_cts_defn} our discretization
will contain terms involving $\div(\Psi_h K_{h,i})$ where
$K_{h,i} \in V_{0h}^J \subset H_0(\div; \Omega)$ 
and $\Psi_h$ is a discrete density reciprocal.
To ensure that $\Psi_h K_{h,i}$ admits a weak divergence we introduce a
finite element space $R_h \subset W^{1, \infty}(\Omega)$ and seek
$\Psi_h \in R_h$ (c.f.~\cref{asm:conc_regularity}).
The regularity $\Psi_h \in W^{1, \infty}(\Omega)$ implies that
$\div(\Psi_h K_{h,i}) \in L^2(\Omega)$.
The equation used to determine $\Psi_h$ is given
below in \cref{eq:full_monolithic_problem_4}.

Recall from \cref{sec:thermo} that constraints must be introduced to
fix the degrees of freedom
$\mu_1^{\textrm{aux}}, \ldots, \mu_n^{\textrm{aux}}, p^{\textrm{aux}}$,
which were absorbed into the enlarged space $P_h' \times (U_h')^n$.
These $n+1$ degrees of freedom require $n+1$ constraints,
denoted very generally by $F(p_h, \bar{\mu}_h, \bar{x}_h) = 0$
where $F : P_h' \times (U_h')^n \times (X_h)^n \to \R^{n+1}$
is a user-chosen function.
For example, integral constraints on the concentrations
(c.f.~\cref{eq:total_moles_constraint})
can be imposed using integrals of
$c_{h,i}$ in \cref{eq:discrete_conc}.
The Gibbs--Duhem relation implies that
only $n$ constraints are required;
we need $n+1$ constraints here since we are solving for all $n$
mole fractions despite only $n-1$ of them being independent.
Hence, one of our constraints is not physical.
Instead, it reflects our desire for the discrete mole fractions to sum to unity.
In particular, satisfaction of \cref{eq:general_const_law_discrete_1}
does not ensure that $\sum_{j=1}^n x_{h,j} = 1$ holds exactly in $\Omega$,
but we can still enforce that this equality holds on average over
$\Omega$, i.e.~we consider
\begin{equation} \label{eq:mfs_constraint}
	\textstyle\int_{\Omega} \big( 1 - \textstyle\sum_{j=1}^n x_{h,j} \big)
	\mathop{\mathrm{d}x} = 0.
\end{equation}
In this work we shall take $F$ to impose \cref{eq:mfs_constraint} along with $n$
physical constraints.

In \cref{sec:discretization} we assumed homogeneous Dirichlet boundary data.
We do not make this assumption here, as the inhomogeneous case has
non-trivial consequences in our monolithic discretization,
discussed further below.
The discrete solution $(v_h, \bar{J}_h) \in V_h$ generally cannot
satisfy the boundary conditions
\cref{eq:sosm_strong2_bcs} exactly. Instead, we consider
\begin{equation} \label{eq:sosm_strong_bcs_discrete}
	v_h = v_{h,D}
	\quad \textrm{and} \quad
	J_{h,i} \cdot n = J_{h, D, i} \cdot n
	\quad \textrm{on} \ \partial \Omega
	\quad \forall i \in \{1, \ldots, n\},
\end{equation}
where $(v_{h,D}, \bar{J}_{h, D}) \in V_h$ are fixed discrete lifting functions,
which, by a method such as interpolation, are chosen to
approximately satisfy the boundary conditions in \cref{eq:sosm_strong2_bcs},
and satisfy the compatibility conditions
$M_i^{-1} \int_{\partial \Omega} J_{h,D,i} \cdot n \mathop{\mathrm{d}s}
= \int_{\Omega} r_i \mathop{\mathrm{d}x}$.

Finally, to state the full monolithic discrete SOSM problem,
we introduce a discrete transport matrix
$\bm{M}_{ij}^{[\bar{c}_h]}$ which is computed using the formula
\cref{eq:transport_m_def} but with the discrete concentrations $\bar{c}_h$
(recall \cref{eq:discrete_conc}).
Since $\bar{c}_h$ is a function of $p_h, \bar{x}_h$, so too is 
$\bm{M}_{ij}^{[\bar{c}_h]}$.
We consider discrete analogues of 
\cref{eq:a_o_cts_defn,eq:a_cts_defn,eq:b_cts_defn}
which employ $\bar{c}_h$ and $\Psi_h$:
\begin{align}
&\begin{cases}
	\begin{split}
		&a_o^{[\bar{c}_h, \Psi_h]} : V_h \times V_{0h} \to \R, \\
		&a_o^{[\bar{c}_h, \Psi_h]}((v_h, \bar{J}_h), (u_h, \bar{K}_h)) :=
		\begin{aligned}[t]
			&\gamma \big( v_h - \Psi_h \textstyle\sum_{j=1}^n J_{h,j},
			u_h - \Psi_h \textstyle\sum_{i=1}^n K_{h,i} \big)_{\Omega} \\
			&+ \textstyle\sum_{i,j=1}^n
			\big( \widetilde{\bm{M}}_{ij}^{[\bar{c}_h]}
			 J_{h,j}, K_{h,i} \big)_{\Omega},
		\end{aligned} \\
		&\textrm{where }
		\widetilde{\bm{M}}_{ij}^{[\bar{c}_h]} := 
		\bm{M}_{ij}^{[\bar{c}_h]} / (M_i M_j c_{h,i} c_{h,j}),
	\end{split}
\end{cases} \\
&\begin{cases} \label{eq:a_discrete_defn}
	\begin{split}
		&a^{[\bar{c}_h, \Psi_h]} : V_h \times V_{0h} \to \R, \\
		&a^{[\bar{c}_h, \Psi_h]}((v_h, \bar{J}_h), (u_h, \bar{K}_h)) 
		:= a_v(v_h, u_h) + a_o^{[\bar{c}_h, \Psi_h]}
		((v_h, \bar{J}_h), (u_h, \bar{K}_h)), \\
		&\textrm{where $a_v$ is defined in \cref{eq:a_v_cts_defn}},
	\end{split}
\end{cases} \\
&\begin{cases} \label{eq:b_psi_h_defn}
	\begin{split}
		&b^{[\Psi_h]} : V_h \times Q_h' \to \R, \\
		&b^{[\Psi_h]}((u_h, \bar{K}_h), (p_h, \bar{\mu}_h)) :=
		\begin{aligned}[t]
			&-(p_h, \div u_h)_{\Omega}
			+ \textstyle{\sum_{i=1}^n} (p_h, \div(\Psi_h K_{h,i}))_{\Omega} \\
			&- \textstyle{\sum_{i=1}^n}
			(M_i^{-1} \mu_{h,i}, \div K_{h,i})_{\Omega}.
		\end{aligned}
	\end{split}
\end{cases}
\end{align}

The monolithic problem is:
find $((v_h, \bar{J}_h), (p_h, \bar{\mu}_h), \bar{x}_h, \Psi_h)
\in V_h \times Q_h' \times (X_h)^n \times R_h$
such that
$(v_h, \bar{J}_h)$
satisfies the boundary conditions in \cref{eq:sosm_strong_bcs_discrete}
and for all test functions
$((u_h, \bar{K}_h), (q_h, \bar{w}_h), \bar{y}_h, r_h)
\in V_{0h} \times Q_h' \times (X_h)^n \times R_h$ there holds
{\allowdisplaybreaks
\begin{subequations} \label{eq:full_monolithic_problem}
	\begin{align}
		a^{[\bar{c}_h, \Psi_h]}((v_h, \bar{J}_h), (u_h, \bar{K}_h))
		+ b^{[\Psi_h]}((u_h, \bar{K}_h), (p_h, \bar{\mu}_h))
		&= \textstyle\sum\limits_{i=1}^n (M_i c_{h,i} f, u_h)_{\Omega},
		\label{eq:full_monolithic_problem_1} \\
		b^{[\Psi_h]}((v_h, \bar{J}_h), (q_h, \bar{w}_h))
		+ b_{dc}^{[\Psi_h]}((v_h, \bar{J}_h), (q_h, \bar{w}_h))
		&= -\textstyle\sum\limits_{i=1}^n (w_{h,i}, r_i)_{\Omega},
		\label{eq:full_monolithic_problem_2} \\
		(\mu_{h,i} - G_i(T, p_h, x_{h,1}, \ldots, x_{h,n}), 
		y_{h,i})_{\Omega}
		&= 0 \quad \forall i \in \{ 1, \ldots, n \},
		\label{eq:full_monolithic_problem_3} \\
		(1 / \Psi_h - \textstyle\sum_{i=1}^n M_i c_{h,i}, r_h)_{\Omega}
		&= 0, \label{eq:full_monolithic_problem_4} \\
		F(p_h, \bar{\mu}_h, \bar{x}_h) &= 0.
		\label{eq:full_monolithic_problem_5}
	\end{align}
\end{subequations}
}%
Here $b_{dc}^{[\Psi_h]}((v_h, \bar{J}_h), (q_h, \bar{w}_h))$
are \textit{density consistency terms}, described below.
As in \cref{sec:discretization}, for $V_h \times Q_h'$ we employ
spaces satisfying
\cref{lemma:discrete_infsup,lemma:discrete_coercivity}.
In practice this leads to spaces of degree $k \geq 1$
being employed for $V_h$ and $k-1$ for $Q_h'$.
In this case we also employ degree $k-1$ spaces for $X_h$ and $R_h$.
This is because $\bar{x}_h, \Psi_h$ are related to $(p_h, \bar{\mu}_h) \in 
Q_h'$ through $L^2$-projections of algebraic equations; we do not expect to 
gain approximation power by using higher degree spaces for these 
unknowns.

To motivate the density consistency terms, recall that 
we enriched the discrete pressure and chemical potential spaces to contain all
constant functions (see \cref{eq:discrete_cp_p_spaces_constants}).
This may have seemed harmless, since at the continuous level the SOSM 
equations \cref{eq:sosm_strong2} are invariant under shifts in the pressure and 
chemical potential by additive constants.
One can also verify that, at the discrete level, shifting $(p_h, \bar{\mu}_h)$ by
constants leaves \cref{eq:full_monolithic_problem_1} unchanged
(integrate by parts and observe $(u_h, \bar{K}_h)$ have vanishing normal 
components on $\partial \Omega$).
However, we must also ensure that using $Q_h'$ instead of $Q_h$
does not affect \cref{eq:full_monolithic_problem_2}, i.e.~that shifting 
$(q_h, \bar{w}_h)$ by constants leaves \cref{eq:full_monolithic_problem_2} 
unchanged.
Since $(v_h, \bar{J}_h)$ do not have vanishing normal components
on $\partial \Omega$ and instead satisfy the boundary conditions in
\cref{eq:sosm_strong_bcs_discrete}, this invariance property
will hold provided we take
\begin{equation} \label{eq:dc_terms}
	 b_{dc}^{[\Psi_h]}((v_h, \bar{J}_h), (q_h, \bar{w}_h))
	 := \textstyle\int_{\partial \Omega} q_h
	 \big( v_h - \textstyle\sum_{i=1}^n \Psi_h J_{h,i} \big) \cdot n
	 \mathop{\mathrm{d}s}.
\end{equation}
We call these density consistency terms, since in general
$v_h \cdot n = \textstyle\sum_{i=1}^n \Psi_h J_{h,i} \cdot n$
holds only approximately on $\partial \Omega$, whereas at 
the continuous level
$v \cdot n = \textstyle\sum_{i=1}^n \Psi J_i \cdot n$ holds exactly
due to the mass-average constraint \cref{eq:massavg}.
We find empirically in \cref{sec:man_sln} that these terms are crucial for 
ensuring that Newton's method converges.
We prefer to use a discrete density reciprocal $\Psi_h$ instead of a discrete 
density $\rho_h$, as this enables the terms in \cref{eq:dc_terms} to be integrated
exactly with quadrature rules of a sufficiently high degree.
However, in our experience discretizing $\rho$ instead of $\Psi$ has no tangible
impact on the performance of our numerical schemes; both choices seem viable 
in practice.

Let $D = \dim (V_{0h} \times Q_h' \times (X_h)^n \times R_h)$.
Once a basis has been chosen, the discrete problem
in \cref{eq:full_monolithic_problem} constitutes $D + n + 1$ equations
in $D$ unknowns.
However, $n +1$ of the equations obtained from
\cref{eq:full_monolithic_problem_2}
can be eliminated, since the compatibility conditions on $J_{h,D,i}$ and our 
construction of $b_{dc}^{[\Psi_h]}$ ensures that
\cref{eq:full_monolithic_problem_2} holds automatically when
$(q_h, \bar{w}_h)$ are constants.
We solve the resulting system of $D$ equations in $D$ unknowns using
Newton's method.
The constraints in \cref{eq:full_monolithic_problem_5} will typically 
(e.g.~when imposing integral constraints) lead to dense rows in the Jacobian
at each Newton iteration;
these could be undesirable when using direct linear solvers.
We avoid these dense rows by employing an auxiliary sparse Jacobian that is 
obtained by replacing the constraints in 
\cref{eq:full_monolithic_problem_5}
with $n+1$ auxiliary constraints that fix the values of $p_h, \bar{\mu}_h$ 
at a mesh node.
The true Jacobian is then a rank $n+1$ update of the auxiliary Jacobian
and we compute the action of its inverse using the Woodbury 
formula \cite{hager1989updating}.

\section{Numerical examples} \label{sec:numerical}

In this section we test our methods.
The experiments were implemented using Firedrake
\cite{FiredrakeUserManual} and PETSc
\cite{petsc-user-ref,dalcinpazklercosimo2011}.
We solved the linear systems using the LU solver MUMPS
\cite{amestoy2001fully}.
Code is available at
{\color{blue}\url{https://bitbucket.org/abaierr/multicomponent_code}}
and software versions are archived on Zenodo \cite{zenodo}.

\subsection{Manufactured solution} \label{sec:man_sln}

To study the spatial errors of the discretization, we consider a manufactured
solution with $n=2$ species.
We conduct tests in two and three spatial dimensions on the domain 
$\Omega = (0, 1)^d$.
The thermodynamic constitutive law is taken to be that of an ideal
gaseous mixture \cite{doble2007perry,guggenheim1967thermodynamics};
hence in \cref{eq:general_const_law} we employ
$G_i(T, p, x_1, x_2) = RT \ln(x_i p)$ and
$V_i(T, p, x_1, x_2) = RT / p$.

We employ the manufactured solution introduced in
\cite[sect.~4.4]{aznaran2024finite}, which requires that $RT=1$.
If the diffusion coefficients are parametrized by $\mathscr{D}_{ij} = D_i D_j$ 
for $D_j > 0$, and $g : \R^d \to \R$ is a smooth function, this solution is 
given by $c_i = \exp(g / D_i)$ and $v_i = D_i \nabla g$.
Specification of $c_i$ and $v_i$ implicitly defines all other unknowns.
We use $D_1 = 1/2, D_2 = 2$.
For $d=2$ we set $g(x, y) = \sin(\pi x) \sin(\pi y)$
and for $d=3$ we set $g(x, y, z) = \sin(\pi x) \sin(\pi y) \sin(\pi z)$.
For the molar masses we use $M_1 = M_2 = 1$, for the viscosities
$\eta = \zeta = 0.1$ and for the augmentation parameter $\gamma = 10$.

We conduct tests for the Picard linearized problem
\cref{eq:discrete_problem} and the nonlinear problem
\cref{eq:full_monolithic_problem}.
In the Picard linearized setting, we use the concentrations
(and quantities $\rho, \Psi$ and $\widetilde{\bm{M}}^{\gamma}_{ij}$)
of the exact solution.
We employ a uniform structured mesh consisting of
triangles when $d=2$ and hexahedra when $d=3$.
For the barycentric velocity and pressure we
use the degree $k=4$ generalized Taylor--Hood pair
\cite[sect.~54.4.1]{ern2021finiteII}.
For the mass flux and chemical potentials we use the
$\mathbb{RT}_k$--$\mathbb{DG}_{k-1}$ pair \cite{raviart1977mixed}
when $d=2$ and its hexahedral analogue \cite{nedelec1980mixed} when $d=3$.
In the nonlinear setting,
we discretize the mole fractions (resp.~density reciprocal) using
discontinuous (resp.~continuous) degree $k-1=3$ polynomial spaces.
As constraints in \cref{eq:full_monolithic_problem_5} we impose
$\int_{\Omega} c_{h,i} \mathop{\mathrm{d}x} = 
\int_{\Omega} c_i \mathop{\mathrm{d}x} \ \forall i$ and
\cref{eq:mfs_constraint}.
We use Newton's method with an absolute tolerance on the residual of $10^{-10}$
in the Euclidean norm.
The initial guess we supply to Newton's method is the $L^2$-projection
of the exact solution, since our purpose here is to study the spatial errors of 
the discretization, not the root-finding ability of Newton's method from a poor
initial guess.
Depending on the mesh size this led to Newton's method
terminating in 1 to 3 iterations.

We uniformly refine the mesh starting from an initial mesh size of $h=2^{-3}$ in 
two dimensions and $h=2^{-1}$ in three dimensions.
We measure the following errors:
\begin{gather*}
	E_v := \norm{v - v_h}_{L^2(\Omega)^d}, \quad
	E_{\nabla v} := \norm{\nabla v - \nabla v_h}_{L^2(\Omega)^{d \times d}}, \quad
	E_p := \norm{p- p_h}_{L^2(\Omega)}, \\
	E_{\bar{J}} := \sqrt{ \textstyle\sum_{i=1}^n 
		\norm{J_i - J_{h,i}}_{L^2(\Omega)^d}^2}, \quad
	E_{\bar{\mu}} := \sqrt{ \textstyle\sum_{i=1}^n 
	\norm{\mu_i - \mu_{h,i}}_{L^2(\Omega)}^2}.
\end{gather*}
We also measure the error in the mass-average constraint
(recall \cref{eq:massavg}) by means of
$E_{\textrm{MA}} := 
\norm{v_h - \Psi \textstyle\sum_{i=1}^n J_{h,i}}_{L^2(\Omega)^d}$
for the Picard linearized problem and
$E_{\textrm{MA}} := 
\norm{v_h - \Psi_h \textstyle\sum_{i=1}^n J_{h,i}}_{L^2(\Omega)^d}$
for the nonlinear problem.
For the nonlinear problem we also measure the mole fraction error
$E_{\bar{x}} := \sqrt{ \textstyle\sum_{i=1}^n 
\norm{x_i - x_{h,i}}_{L^2(\Omega)}^2}$.

\begin{table} \label{table:picard_linearized}
\setlength{\tabcolsep}{3pt}
\renewcommand{\arraystretch}{1.25}
\caption{Computed spatial errors and associated rates
for the Picard linearized test case in \cref{sec:man_sln}.}
\begin{center}
\footnotesize
\begin{tabular}{ccccccc}
$h$ & $E_v$ (rate) & $E_{\nabla v}$ (rate) & $E_p$ (rate)
& $E_{\bar{J}}$ (rate) & $E_{\bar{\mu}}$ (rate) & $E_{\textrm{MA}}$ (rate)
\\ \Xhline{2\arrayrulewidth}
\multicolumn{7}{l}
{Spatial dimension $d=2$}
\\
\makecell{\ $2^{-3}$ \\ \ $2^{-4}$ \\ \ $2^{-5}$ \\ \ $2^{-6}$ \\ \ $2^{-7}$}
& \makecell[l]
{1.8e-05 \\ 5.2e-07 (5.1) \\ 1.6e-08 (5.0) \\ 5.0e-10 (5.0) \\ 1.6e-11 (4.9)}
& \makecell[l]
{1.9e-3 \\ 1.1e-4 (4.2) \\ 6.3e-6 (4.1) \\ 3.9e-7 (4.0) \\ 2.6e-8 (3.9)}
& \makecell[l]
{4.4e-4 \\ 2.6e-5 (4.1) \\ 1.6e-6 (4.0) \\ 9.8e-8 (4.0) \\ 1.0e-8 (3.3)}
& \makecell[l]
{5.0e-4 \\ 3.0e-5 (4.1) \\ 1.8e-6 (4.0) \\ 1.1e-7 (4.0) \\ 7.1e-9 (4.0)}
& \makecell[l]
{1.0e-4 \\ 5.5e-6 (4.2) \\ 3.2e-7 (4.1) \\ 1.9e-8 (4.0) \\ 5.9e-9 (1.7)}
& \makecell[l]
{1.5e-4 \\ 9.1e-6 (4.1) \\ 5.6e-7 (4.0) \\ 3.5e-8 (4.0) \\ 2.2e-9 (4.0)}
\\
\multicolumn{7}{l}{Spatial dimension $d=3$}
\\
\makecell{\ $2^{-1}$ \\ \ $2^{-2}$ \\ \ $2^{-3}$ \\ \ $2^{-4}$}
& \makecell[l]
{4.6e-03 \\ 1.1e-04 (5.4) \\ 3.3e-06 (5.0) \\ 9.5e-08 (5.1)}
& \makecell[l]
{1.1e-1 \\ 5.2e-3 (4.3) \\ 3.9e-4 (3.7) \\ 2.0e-5 (4.3)}
& \makecell[l]
{3.6e-2 \\ 1.7e-3 (4.4) \\ 1.4e-4 (3.6) \\ 7.9e-6 (4.1)}
& \makecell[l]
{5.7e-2 \\ 2.4e-3 (4.6) \\ 2.1e-4 (3.5) \\ 1.3e-5 (4.1)}
& \makecell[l]
{9.8e-3 \\ 3.4e-4 (4.9) \\ 1.7e-5 (4.4) \\ 6.2e-7 (4.7)}
& \makecell[l]
{1.8e-2 \\ 9.2e-4 (4.3) \\ 6.1e-5 (3.9) \\ 3.7e-6 (4.0)}
\\
\end{tabular}
\end{center}
\end{table}

The numerical results for the Picard linearized problem are shown in
\Cref{table:picard_linearized}.
Our chosen finite element spaces satisfy the hypotheses of
\cref{lemma:discrete_infsup,lemma:discrete_coercivity}, and
the quasi-optimal error bound in \cref{eq:discrete_quasioptimal} therefore
predicts that
$E_v, E_{\nabla v}, E_p, E_{\bar{J}}, E_{\bar{\mu}}$ and $E_{\textrm{MA}}$
go to zero as
$\mathcal{O}(h^k) = \mathcal{O}(h^4)$.
This theoretical prediction is consistent with the empirical convergence rates
shown in \Cref{table:picard_linearized}.
When $d=2$ and $h=2^{-7}$ the rates for $E_p$ and $E_{\bar{\mu}}$ appear to decay
below $\mathcal{O}(h^4)$, but we believe this is due to rounding error and solver tolerances.
Moreover, it appears that $E_v$ converges
as $\mathcal{O}(h^{k+1}) = \mathcal{O}(h^5)$, and we hypothesize that it may be
possible to prove this rigorously using an Aubin--Nitsche technique.

\begin{table} \label{table:full_nonlinear}
\setlength{\tabcolsep}{3pt}
\renewcommand{\arraystretch}{1.25}
\caption{Computed spatial errors and associated rates
for the nonlinear test case in \cref{sec:man_sln}.}
\begin{center}
\footnotesize
\begin{tabular}{cccccccc}
$h$ & $E_v$ (rate) & $E_{\nabla v}$ (rate) & $E_p$ (rate)
& $E_{\bar{J}}$ (rate) & $E_{\bar{\mu}}$ (rate) & $E_{\textrm{MA}}$ (rate)
& $E_{\bar{x}}$ (rate)
\\ \Xhline{2\arrayrulewidth}
\multicolumn{7}{l}
{Spatial dimension $d=2$}
\\
\makecell{\ $2^{-3}$ \\ \ $2^{-4}$ \\ \ $2^{-5}$ \\ \ $2^{-6}$ \\ \ $2^{-7}$}
& \makecell[l]
{2.0e-05 \\ 6.5e-07 (5.0) \\ 2.7e-08 (4.6) \\ 1.4e-09 (4.3) \\ 8.2e-11 (4.1)}
& \makecell[l]
{2.0e-3 \\ 1.3e-4 (4.0) \\ 1.1e-5 (3.5) \\ 1.2e-6 (3.2) \\ 1.5e-7 (3.1)}
& \makecell[l]
{4.6e-4 \\ 3.0e-5 (3.9) \\ 2.5e-6 (3.6) \\ 2.6e-7 (3.3) \\ 3.1e-8 (3.1)}
& \makecell[l]
{5.3e-4 \\ 3.5e-5 (3.9) \\ 3.0e-6 (3.5) \\ 3.3e-7 (3.2) \\ 3.9e-8 (3.1)}
& \makecell[l]
{1.0e-4 \\ 8.3e-6 (3.6) \\ 9.2e-7 (3.2) \\ 1.1e-7 (3.0) \\ 1.4e-8 (3.0)}
& \makecell[l]
{2.1e-4 \\ 1.6e-5 (3.8) \\ 1.5e-6 (3.4) \\ 1.7e-7 (3.2) \\ 2.0e-8 (3.0)}
& \makecell[l]
{9.4e-06 \\ 5.7e-07 (4.0) \\ 3.6e-08 (4.0) \\ 2.2e-09 (4.0) \\ 1.4e-10 (4.0)}
\\
\multicolumn{7}{l}{Spatial dimension $d=3$}
\\
\makecell{\ $2^{-1}$ \\ \ $2^{-2}$ \\ \ $2^{-3}$ \\ \ $2^{-4}$}
& \makecell[l]
{5.9e-03 \\ 1.4e-04 (5.4) \\ 3.7e-06 (5.3) \\ 1.1e-07 (5.1)}
& \makecell[l]
{8.4e-2 \\ 6.1e-3 (3.8) \\ 4.1e-4 (3.9) \\ 2.0e-5 (4.3)}
& \makecell[l]
{2.5e-2 \\ 1.7e-3 (3.9) \\ 1.4e-4 (3.6) \\ 7.9e-6 (4.1)}
& \makecell[l]
{6.0e-2 \\ 2.5e-3 (4.6) \\ 2.1e-4 (3.6) \\ 1.3e-5 (4.1)}
& \makecell[l]
{7.1e-3 \\ 3.3e-4 (4.5) \\ 1.7e-5 (4.3) \\ 6.3e-7 (4.7)}
& \makecell[l]
{1.4e-2 \\ 1.2e-3 (3.6) \\ 7.5e-5 (4.0) \\ 4.7e-6 (4.0)}
& \makecell[l]
{1.3e-03 \\ 6.8e-05 (4.2) \\ 3.5e-06 (4.3) \\ 2.1e-07 (4.0)}
\\
\end{tabular}
\end{center}
\end{table}

The results for the nonlinear problem are shown in
\Cref{table:full_nonlinear}.
When $d=2$ it appears that 
$E_v, E_{\nabla v}, E_p, E_{\bar{J}}, E_{\bar{\mu}}$ and $E_{\textrm{MA}}$
converge suboptimally by a factor of $h$, i.e.~they converge as
$\mathcal{O}(h^{k-1}) = \mathcal{O}(h^3)$ ($\mathcal{O}(h^4)$ for $E_v$).
When $d=3$ these quantities appear to instead converge as 
$\mathcal{O}(h^4)$ ($\mathcal{O}(h^5)$ for $E_v$), but we anticipate that if
the mesh were to be refined further, they would also begin to converge suboptimally
by a factor of $h$.
Interestingly, the mole fraction error $E_{\bar{x}}$ 
seems to converge optimally as $\mathcal{O}(h^4)$ for $d=2,3$
despite the other quantities converging suboptimally when $d=2$.

We hypothesize that the suboptimal rates in
$E_v, E_{\nabla v}, E_p, E_{\bar{J}}, E_{\bar{\mu}}$ and $E_{\textrm{MA}}$,
which only appear when solving the nonlinear problem,
are due to the presence of $\nabla \Psi_h$ in our discretization. 
More precisely, the form $b^{[\Psi_h]}$ in
\cref{eq:b_psi_h_defn} contains the terms
\begin{equation} \label{eq:problematic_term}
	\big(p_h, \div(\Psi_h K_{h,i})\big)_{\Omega} = 
	\big(p_h, (\nabla \Psi_h) \cdot K_{h,i}\big)_{\Omega}
	+ \big(p_h, \Psi_h \div K_{h,i}\big)_{\Omega}.
\end{equation}
The culprit seems to be the term
$\big(p_h, (\nabla \Psi_h) \cdot K_{h,i}\big)_{\Omega}$.
In additional experiments (not shown here), we find that if
$\nabla \Psi_h$ is replaced by $\nabla \Psi$ in this term, we recover optimal 
rates of convergence for the nonlinear problem.
Of course, this is useless in practice; we only know what $\nabla \Psi$ is here
because we are using a manufactured solution.
However, this does provide a heuristic explanation for these suboptimal rates.
We expect that $\nabla \Psi_h$ converges to $\nabla \Psi$
slower by a factor of $h$ than $\Psi_h$ converges to $\Psi$,
which ``pollutes'' $b^{[\Psi_h]}$.
Also, we expect that $\Psi_h$ can, at best, converge at the same rate that
$p_h, \bar{\mu}_h$ converge, since the concentrations
(hence also $\Psi$) are related to $p, \bar{\mu}$ through algebraic 
equations \cref{eq:general_const_law}.
The ``pollution'' in $b^{[\Psi_h]}$ caused by $\nabla \Psi_h$ therefore, at best,
exceeds the optimal approximation error for $p, \bar{\mu}$ by a factor of $h$, 
leading to suboptimal rates.
Discretizing the density $\rho$ instead of its reciprocal 
$\Psi$ does not appear to alleviate these suboptimal rates,
and more broadly it does not seem straightforward to devise a numerical scheme 
that circumvents this difficulty.
The terms in \cref{eq:problematic_term} arise due to coupling of the Stokes 
and OSM equations;
they would not appear in a formulation of the isobaric OSM problem alone.

The density consistency terms in
\cref{eq:full_monolithic_problem_2}
are crucial in this example.
Without them, we find that Newton's method diverges
for the values of $d$ and $h$ from \Cref{table:full_nonlinear}.

\subsection{Benzene-cyclohexane} \label{sec:benzene_cyclohexane}

We consider the three-dimensional mixing of liquid flows of benzene and
cyclohexane in a microfluidic container.
This example is analogous to the two-dimensional benzene-cyclohexane simulation 
in \cite[sect.~4.5]{aznaran2024finite}, but here we consider the more challenging
three-dimensional case.
We refer to benzene and cyclohexane as species 1 and 2, respectively.
Following \cite{aznaran2024finite} we use parameter values
$T=298.15 \ \textrm{K}$,
$\mathscr{D}_{12} = 2.1 \times 10^{-9} \ \textrm{m}^2 / \textrm{s}$,
$\eta = 6 \times 10^{-4} \ \textrm{Pa} \cdot \textrm{s}$
and
$\zeta = 10^{-7} \ \textrm{Pa} \cdot \textrm{s}$.
As molar masses we take $M_1 = 0.078 \ \textrm{kg} / \textrm{mol}$ and
$M_2 = 0.084 \ \textrm{kg} / \textrm{mol}$.

We take the ambient pressure to be
$p^{\textrm{ref}} = 10^5 \ \textrm{Pa}$
and assume that the dependence of the partial molar volumes on
pressure variations and composition can be neglected.
Hence each $V_i(T, p, x_1, x_2)$ is treated as being a constant
(see e.g.~\cite{druet2021global} and references therein).
It follows that $V_i(T, p, x_1, x_2) = 1 / c_i^{\textrm{ref}}$
where $c_i^{\textrm{ref}}$ is the concentration of pure species $i$
at temperature $T$ and pressure $p^{\textrm{ref}}$.
Moreover, benzene and cyclohexane form a non-ideal solution, and
we capture this using a Margules model \cite{doble2007perry} of the form
\begin{align*}
	G_1(T, p, x_1, x_2) &=
	p / c_1^{\textrm{ref}}
	+ RT \ln x_1
	+ RT x_2^2 (A_{12} + 2 (A_{21} - A_{12}) x_1), \\
	G_2(T, p, x_1, x_2) &=
	p / c_2^{\textrm{ref}}
	+ RT \ln x_2
	+ RT x_1^2 (A_{21} + 2 (A_{12} - A_{21}) x_2).
\end{align*}
As in \cite{aznaran2024finite} we set
$A_{12} = 0.4498$,
$A_{21} = 0.4952$,
$M_1 c_1^{\textrm{ref}} = 876 \ \textrm{kg} / \textrm{m}^3$
and
$M_2 c_2^{\textrm{ref}} = 773 \ \textrm{kg} / \textrm{m}^3$.
Since each $G_i$ is linear in $p$, and $V_i$ is independent of $p$,
the value of $p^{\textrm{aux}}$ plays no role in the constitutive relation
\cref{eq:general_const_law} as it can be absorbed into each of the
$\mu_i^{\textrm{aux}}$.
The auxiliary constants in \cref{eq:general_const_law} therefore only
allow us to prescribe $n-1 = 1$ physical constraint.
To demonstrate the flexibility of our method we take this to be
that $M_1 c_{h,1} - M_2 c_{h,2}$ integrates to zero on the outflow boundary
of the container; physically this states that the species have equal average 
densities on the outflow.
The non-physical constraints that we employ are
$\int_{\Omega} p \mathop{\mathrm{d}x} = 0$ and
\cref{eq:mfs_constraint}.

The domain $\Omega$, shown in \Cref{fig:mesh_pressure}, is a chamber with two 
inflow pipes and an outflow pipe.
The length scale of the domain is on the order of millimetres.
Benzene flows in from the back pipe (at $x=-0.5$),
cyclohexane from the side pipe (at roughly $x=1$),
and both flow out of the front pipe (at $x=5.5$).
As boundary conditions, we enforce parabolic profiles on $J_i \cdot n$
at inflow $i$ and on the outflow, and $J_i \cdot n = 0$ elsewhere on the 
boundary.
Instead of specifying the value of the barycentric velocity on the
inflows and outflow, we enforce $\rho v \cdot n = (J_1 + J_2) \cdot 
n$ and $\rho v \times n = 0$ in these regions.
Elsewhere on the boundary we enforce $v = 0$.
The magnitude of the parabolic profile enforced on $J_i \cdot n$ is 
$M_i c_i^{\textrm{ref}} v_i^{\textrm{ref}}$
where, as in \cite{aznaran2024finite},
we symmetrize the molar fluxes so that 
$c_1^{\textrm{ref}} v_1^{\textrm{ref}} = 
c_2^{\textrm{ref}} v_2^{\textrm{ref}}$.
We choose $v_1^{\textrm{ref}} = 10 \ \textrm{{\textmu}m} / \textrm{s}$,
which results in a Reynolds number for the problem of
$\textrm{Re} = 3 \times 10^{-2}$ and a Péclet number of
$\textrm{Pe} = 1 \times 10^{1}$.

\begin{figure}[h] \label{fig:mesh_pressure}
\caption{The domain and mesh used in \cref{sec:benzene_cyclohexane}.
A unit of length on the axes corresponds to a physical length of 2mm.
The volume is colored by the nondimensionalized pressure solution $p$.}
\vspace{0.25em}
\centering
\includegraphics[width=1.0\textwidth]{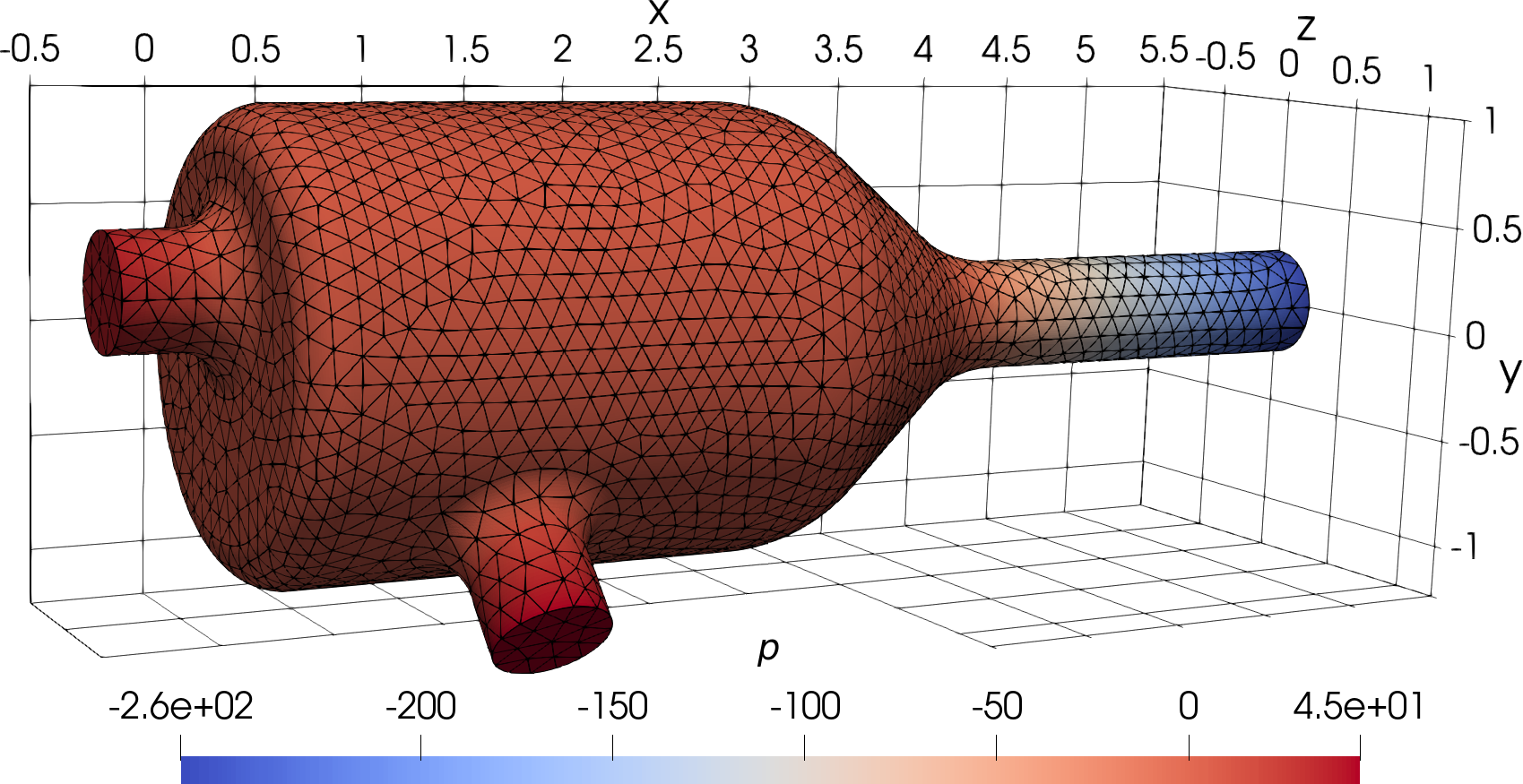}
\end{figure}

We employ a curved tetrahedral mesh of order $5$, shown in
\Cref{fig:mesh_pressure}.
The mesh was generated using ngsPETSc \cite{ngspetsc,schoberl1997netgen}
and OpenCASCADE Technology.
For the barycentric velocity and pressure we
use the degree $k=5$ generalized Taylor--Hood pair
\cite[sect.~54.4.1]{ern2021finiteII}.
For the mass flux and chemical potentials we use the
$\mathbb{BDM}_k$--$\mathbb{DG}_{k-1}$ pair \cite{brezzi1985two,nedelec1986new}.
For the mole fractions (resp.~density reciprocal) we use
discontinuous (resp.~continuous) degree $k-1=4$ polynomial spaces.
We solve a nondimensionalized SOSM system with augmentation parameter 
$\gamma = 0.1$.

We solve the discretized nonlinear problem using Newton's method
with an absolute tolerance on the residual of $10^{-11}$ in the Euclidean norm.
We first solve the problem with a low-order $k=2$ discretization
(on the same curved mesh), and we use the resulting solution as the
initial guess to Newton's method in the high-order $k=5$ case.
With this low-order initial guess,
Newton's method in the $k=5$ case converges in 5 iterations.
We also solve the $k=2$ problem using Newton's method with the same absolute
tolerance, but we employ continuation \cite{deuflhard2011newton}
to aid convergence.
We take five continuation steps by gradually increasing
$v_1^{\textrm{ref}}$. In particular we use values of
$v_1^{\textrm{ref}} = 10^{\theta_{\textrm{cont}} / 4}
\ \textrm{{\textmu}m} / \textrm{s}$
for $\theta_{\textrm{cont}} = 0, 1, 2, 3, 4$.
For our initial guess at the first step $\theta_{\textrm{cont}} = 0$, 
we set the velocity and mass fluxes to be zero,
and assume an equimolar isobaric mixture wherein
$x_{h,1} = x_{h,2} = 0.5$ and $p_h = 0$,
which together with \cref{eq:general_const_law}
determines the $\mu_{h, i}$ and $\Psi_h$.
The five continuation steps $\theta_{\textrm{cont}} = 0, 1, 2, 3, 4$ 
respectively required $5, 5, 6, 6, 7$ Newton iterations.

\begin{figure}[h] \label{fig:species_1}
\caption{A cross-sectional view of the domain in 
\cref{sec:benzene_cyclohexane}.
The volume is colored by the mole fraction $x_1$ of benzene.
Isosurfaces of $x_1$ and streamlines of the nondimensionalized
benzene velocity field $v_1$
are also shown.
The streamlines are colored by the magnitude of $v_1$.}
\vspace{0.25em}
\centering
\includegraphics[width=1.0\textwidth]{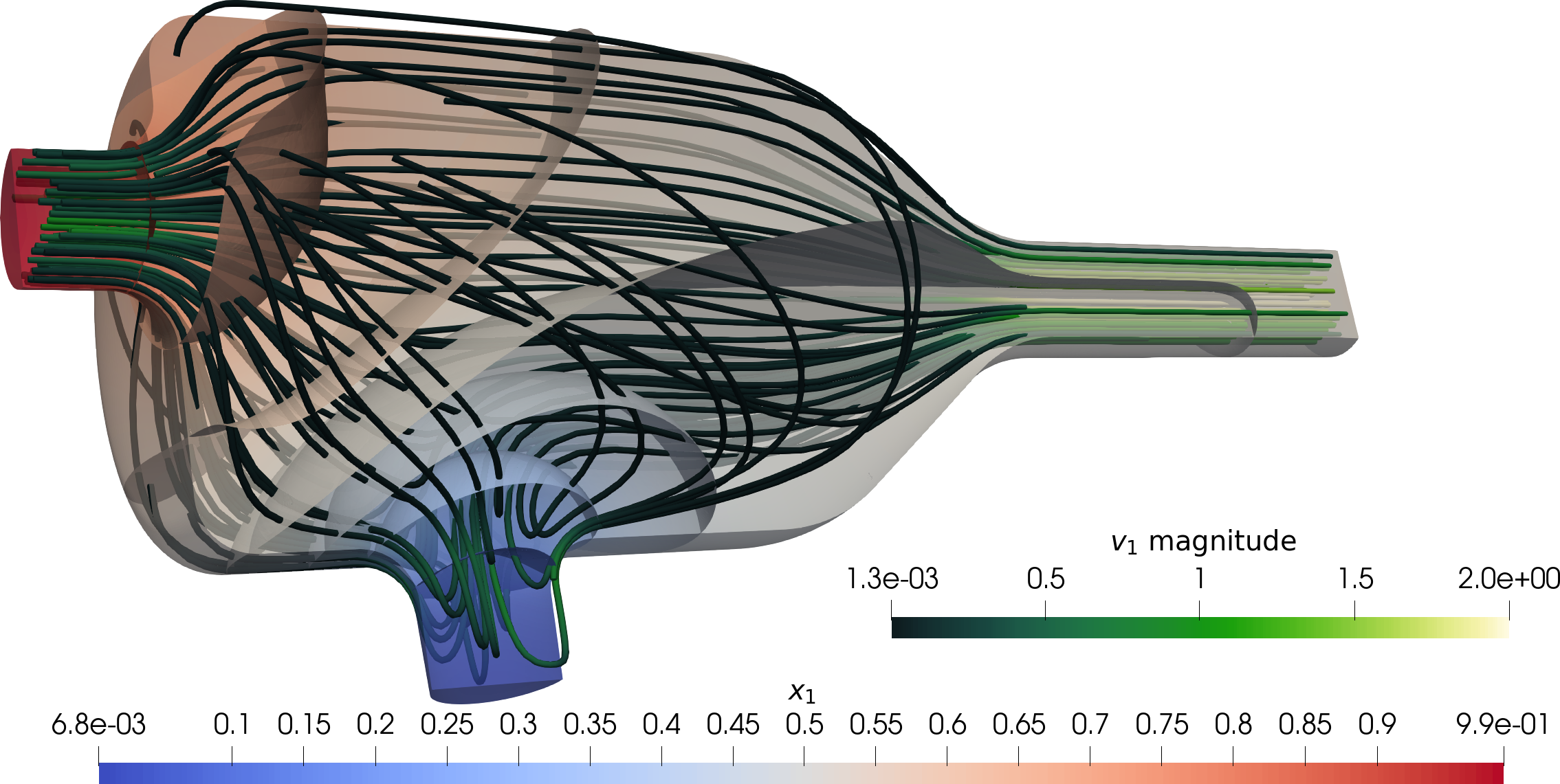}
\end{figure}

The pressure solution, shown in \Cref{fig:mesh_pressure},
is largest at the inlets and at a minimum on the outlet.
The benzene mole fraction and velocity are visualized in \Cref{fig:species_1}.
The benzene mole fraction approaches 1 at the benzene inlet
and a value close to 0.5 at the outlet.
The benzene velocity field is laminar but
exhibits complicated behavior near the cyclohexane inlet; 
our high-order scheme captures this effectively despite the coarse mesh
(the cell diameter is roughly 1/4 the inlet diameter).
Finally, we compute errors in the (nondimensionalized) mass-average 
constraint and in the requirement that the mole fractions sum to unity, 
yielding reassuringly small values of
$\norm{v_h - \Psi_h \textstyle\sum_{i=1}^n J_{h,i}}_{L^2(\Omega)^d} 
= 1.1 \times 10^{-4}$
and
$\norm{1 - \textstyle\sum_{i=1}^n x_{h,i}}_{L^2(\Omega)} 
= 1.6 \times 10^{-6}$.

\subsection{High-order versus low-order discretization} \label{ss:high_vs_low}

To study the advantages of using a high-order discretization,
we repeat the benzene-cyclohexane simulation in \cref{sec:benzene_cyclohexane},
but we vary the polynomial degree $k$ of the finite element spaces.
In particular we vary $k \in \{2, \ldots, 5 \}$, and we compare this to fixing the
degree at $k=2$ and once uniformly refining the mesh (starting with the same mesh 
from \cref{sec:benzene_cyclohexane}).
In all cases we employ a curved mesh of order 5, so that the effect of curving the 
mesh with different orders does not play a role in our experiment.

\begin{figure}[h] \label{fig:high_vs_low}
	\caption{Plots of the $L^2$-errors in the mole fraction and
	(nondimensionalized) mass-average constraints for the experiment of
	\cref{ss:high_vs_low}.
	The blue plot with square markers is obtained by holding the mesh fixed
	and varying $k \in \{2, \ldots, 5 \}$. The red plot with circle markers
	is obtained by holding $k=2$ fixed and once uniformly refining the mesh.}
	\centering
	\includegraphics[width=1.0\textwidth]{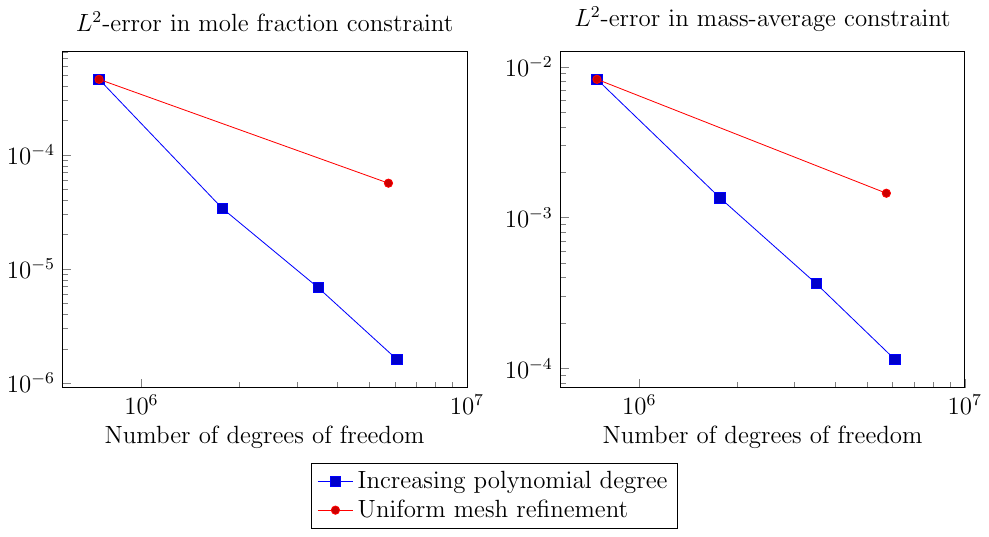}
\end{figure}

To quantify the accuracy of our numerical solutions, we measure 
$L^2$-errors in the mole fraction constraint
$\norm{1 - \textstyle\sum_{i=1}^n x_{h,i}}_{L^2(\Omega)}$
and (nondimensionalized) mass-average constraint
$\norm{v_h - \Psi_h \textstyle\sum_{i=1}^n J_{h,i}}_{L^2(\Omega)^d}$.
In \Cref{fig:high_vs_low} we plot these errors against the
number of degrees of freedom (DOFs) of the finite element spaces.
In blue we plot the approach of holding the mesh fixed and varying
$k \in \{2, \ldots, 5 \}$, and in red the approach of holding $k=2$ fixed
and uniformly refining the mesh.

When $k=2$ on the coarse mesh, the finite element space has $7.4 \times 10^5$
DOFs, while on the fine mesh for $k=2$ there are 
$5.7 \times 10^6$ DOFs.
When $k=5$ on the coarse mesh there are $6.1 \times 10^6$ DOFs,
which is comparable to that of $k=2$ on the fine mesh.
However, we see from \Cref{fig:high_vs_low} that the $k=5$ discretization
achieves significantly lower $L^2$-errors than the $k=2$ discretization on the 
fine mesh.
In fact, even taking $k \in \{3, 4\}$ leads to a greater accuracy than the 
$k=2$ discretization on the fine mesh, and for a much lower DOF count.
These findings are corroborated in \Cref{fig:high_vs_low_2} where, 
taking the $k=5$ solution as a reference solution,
the errors in all fields (against this reference solution)
are again lower for $k \in \{3, 4\}$ on the
coarse mesh than for $k=2$ on the refined mesh.
These findings indicate that a high-order
discretization may be a more efficient choice for a given DOF count,
and motivate the future study of preconditioners for this case.

\begin{figure}[h] \label{fig:high_vs_low_2}
	\caption{Plots of the errors in (nondimensionalized) fields for the 	
		experiment of \cref{ss:high_vs_low}. Errors are computed with respect
		to the $k=5$ solution on the coarse mesh.
		The blue plots with square markers are obtained by holding the mesh fixed
		and varying $k \in \{2, \ldots, 4 \}$. The red plots with circle markers
		are obtained by holding $k=2$ fixed and once uniformly refining the mesh.}
	\centering
	\vspace{1mm}
	\includegraphics[width=1.0\textwidth]
	{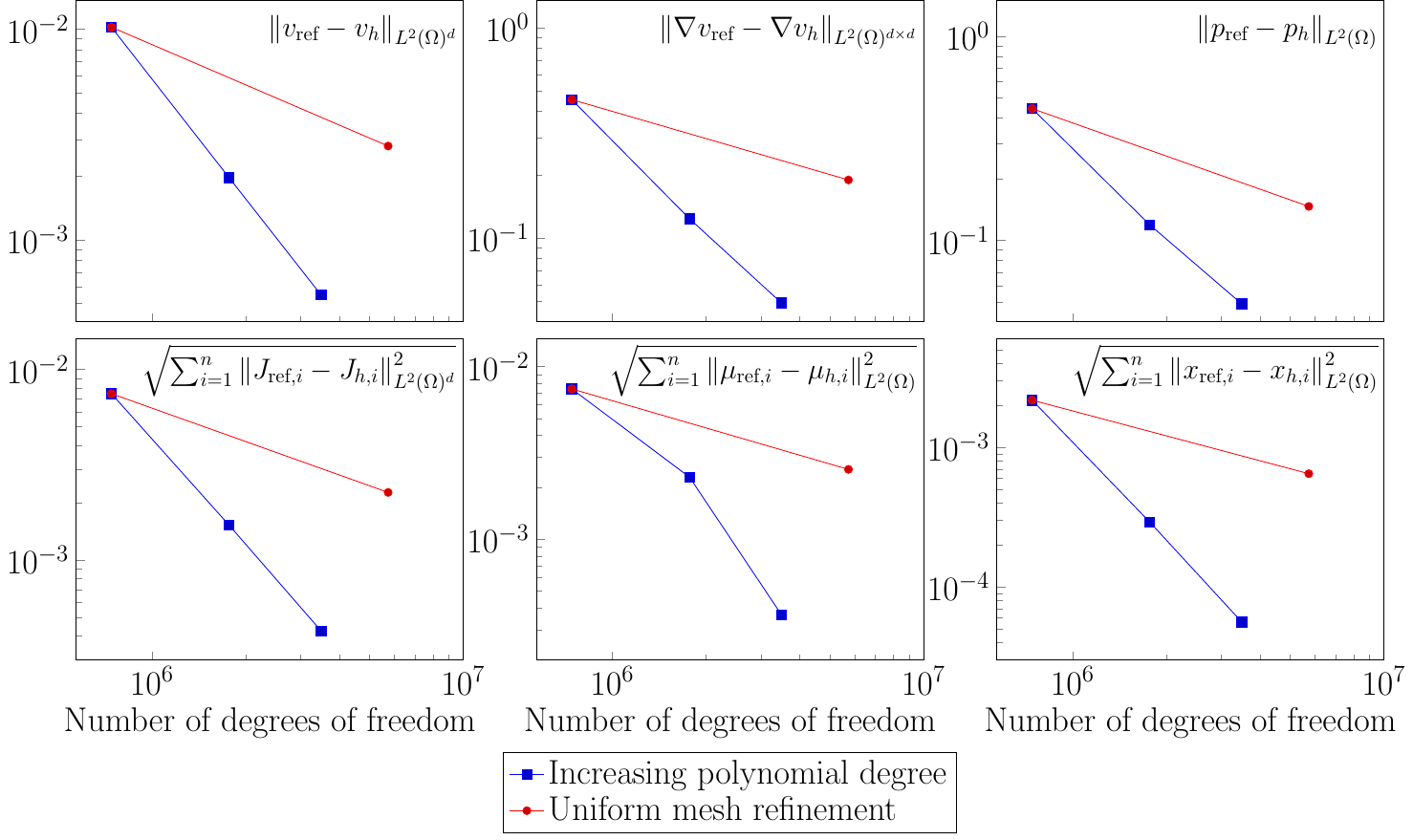}
\end{figure}

\subsection{Comparison to a previous method} \label{ss:comparison}

We next compare our method to that of \cite{aznaran2024finite},
which, to the best of our knowledge, is the only other finite element method in 
the
literature for the SOSM equations in the non-ideal setting.
Since the discretization of \cite{aznaran2024finite} is limited to two-dimensional
low-order simulations, we use as a basis for comparison the two-dimensional
benzene-cyclohexane simulation from \cite{aznaran2024finite}.
The physical setup is similar to that of \cref{sec:benzene_cyclohexane} and we
refer to \cite[sect.~4.5]{aznaran2024finite} for the details.

For the low-order method of \cite{aznaran2024finite} we employ a 
non-curved mesh consisting of 10503 triangles.
For our method we employ a curved mesh of order 4 with 3164 triangles
and we employ a discretization of order $k=4$.
We have chosen the resolution of these meshes so that when we compare the two
methods they result in linear systems of approximately the same size; 
see below.
In both cases we use ngsPETSc \cite{ngspetsc,schoberl1997netgen} to mesh the  
domain.
The method of \cite{aznaran2024finite} employs Picard iteration, which, due to our
choice of mesh, requires that a linear system with $3.49 \times 10^5$ DOFs be 
solved at each iteration.
For our method we employ the monolithic Newton scheme (without continuation) as 
described in \cref{sec:newton}, for which the resulting Newton linearized systems 
to be solved at each iteration have $3.51 \times 10^5$ DOFs.
In both cases we choose an equimolar mixture as an initial guess in the iterative 
schemes.

The two discretizations that we consider result in different nonlinear systems of 
equations to be solved.
The nonlinear equations are solved using Picard iteration in 
\cite{aznaran2024finite} and Newton iteration for our method.
To quantify how well these iterative schemes converge,
we measure the $L^2$-norm of the update in the (nondimensionalized) concentration 
fields at each iteration, and we terminate once this falls below $10^{-10}$.
We have chosen to measure this quantity because it is readily computable for both 
of the methods, despite the fact that they are based on different formulations 
and choices of unknowns in the SOSM equations.
We also measure $L^2$-errors in the (nondimensionalized) mass-average 
constraint at each iteration.

\begin{figure}[h] \label{fig:comparison_easy}
\caption{Plot of the concentration updates and mass-average constraint errors
in the $L^2$-norm, versus the number of iterations, for the experiment of 
\cref{ss:comparison} with
$v_1^{\textrm{ref}} = 0.4 \ \textrm{{\textmu}m} / \textrm{s}$.
We compare our method with the previous method of \cite{aznaran2024finite}.
The dashed grey line represents the termination
condition of $10^{-10}$ for the concentration updates in the $L^2$-norm.
}
\centering
\includegraphics[width=1.0\textwidth]{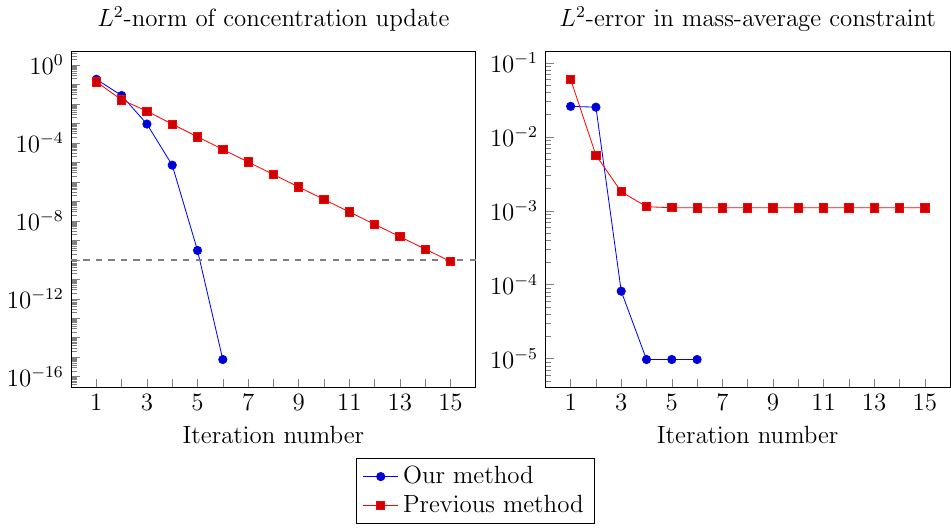}
\end{figure}

We first carry out this experiment with a benzene reference inflow velocity of
$v_1^{\textrm{ref}} = 0.4 \ \textrm{{\textmu}m} / \textrm{s}$.
The results of this experiment are plotted in \Cref{fig:comparison_easy}.
The concentration update plot indicates that both methods successfully converge, 
with our method doing so quadratically (as expected of Newton's method),
and the method of \cite{aznaran2024finite} linearly (as expected of Picard 
iteration).
The quadratic convergence of Newton's method means that our scheme converges much 
faster than that of \cite{aznaran2024finite}; we converge in 6 iterations whereas 
\cite{aznaran2024finite} converges in 15.
Moreover, our high-order discretization achieves a significantly lower $L^2$-error 
in the mass-average constraint, despite the fact that both schemes result in 
linear systems with roughly the same number of DOFs.
We ran the experiment of this subsection in parallel on 8 cores with an
Intel Core i9-10920X processor, and we solved the linear systems with
MUMPS \cite{amestoy2001fully}.
The linear solve runtime was similar for both methods.
For our method the average runtime was 2.0 seconds, where we have accounted for 
the additional computational time resulting from applying the Woodbury formula
(recall \cref{sec:newton}).
For the method of \cite{aznaran2024finite} the average runtime was 2.4 seconds.

Finally, we repeat this experiment with a larger benzene reference 
inflow velocity of
$v_1^{\textrm{ref}} = 4.0 \ \textrm{{\textmu}m} / \textrm{s}$.
This results in a nonlinear system that is appreciably more difficult to solve,
as demonstrated in \cref{fig:comparison_hard}.
Indeed, our method converges but now requires 8 iterations to do so.
Moreover, for the method of \cite{aznaran2024finite} we employ the 
under-relaxation strategy outlined in 
\cite[sect.~4.5]{aznaran2024finite} to aid with nonlinear convergence,
but we find that for both small and moderate under-relaxation parameters 
$\epsilon \in \{ 0.1, 0.5 \}$
the method fails to converge and the mass-average error
does not decrease with the iteration count.
These findings indicate that our Newton scheme is more robust than Picard 
iteration when solving more challenging problems.
In this case the linear solve runtimes were again similar for both methods;
we recorded an average runtime of 1.9 seconds for our method and 2.4 seconds 
for 
that of \cite{aznaran2024finite}.

\begin{figure}[h] \label{fig:comparison_hard}
	\caption{Plot of the concentration updates and mass-average constraint errors
		in the $L^2$-norm, versus the number of iterations, for the experiment of 
		\cref{ss:comparison} with
		$v_1^{\textrm{ref}} = 4.0 \ \textrm{{\textmu}m} / \textrm{s}$.
		We compare our method with the previous method of \cite{aznaran2024finite} 
		using an under-relaxation parameter $\epsilon$.
		The dashed grey line represents the termination
		condition of $10^{-10}$ for the concentration updates in the $L^2$-norm.}
	\centering
	\includegraphics[width=1.0\textwidth]{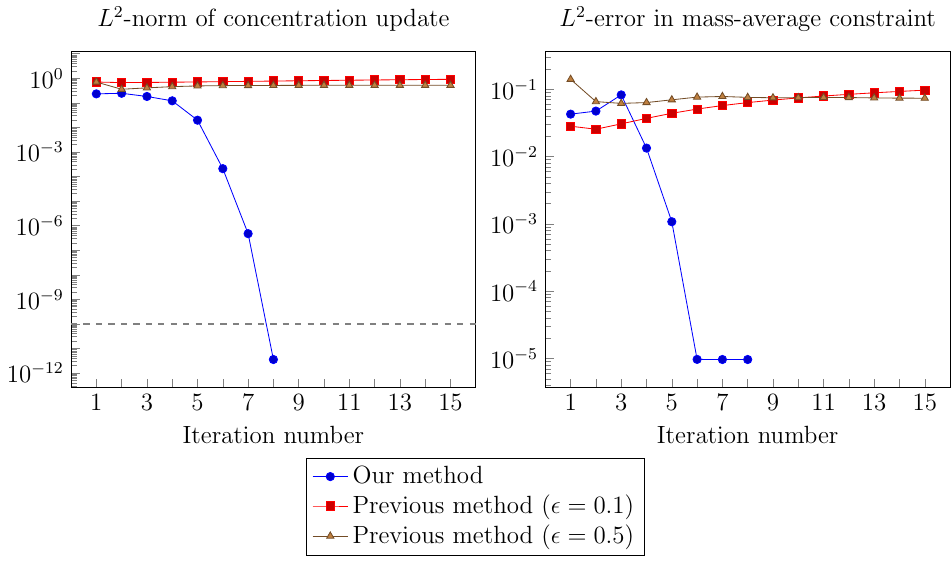}
\end{figure}

\section{Conclusions} \label{sec:conclusions}

In this paper we have introduced and analyzed a large family of finite element 
methods for solving the SOSM equations, which model bulk momentum transport and 
multicomponent diffusion in concentrated mixtures.
To the best of our knowledge, this is the first paper in the finite element
literature that introduces a discretization for these equations that is applicable 
to non-ideal mixtures and is straightforward to implement in two and three spatial 
dimensions.
At a theoretical level, we studied a Picard linearization of the SOSM problem
and we proved that our discretizations are convergent and quasi-optimal in this 
linearized setting.
We also showed how the methods can be extended to the full nonlinear SOSM problem 
and how the resulting monolithic system can be solved using Newton's method.
Numerical experiments substantiated our theory and demonstrate that the 
methods achieve high-order spatial convergence rates.
Our numerical experiments also suggest that, for a given computational cost,
our high-order discretizations in conjunction with Newton's method outperform 
previous low-order approaches in both accuracy and robustness.

\bibliographystyle{siamplain}

\begin{thebibliography}{10}
	
	\bibitem{amestoy2001fully}
	{\sc P.~R. Amestoy, I.~S. Duff, J.-Y. L'Excellent, and J.~Koster}, {\em A fully
		asynchronous multifrontal solver using distributed dynamic scheduling}, 
		SIAM
	J. Matrix Anal. Appl., 23 (2001), pp.~15--41.
	
	\bibitem{aznaran2022transformations}
	{\sc F.~R. Aznaran, P.~E. Farrell, and R.~C. Kirby}, {\em Transformations for
		{P}iola-mapped elements}, SMAI J. Comput. Math., 8 (2022), pp.~399--437.
	
	\bibitem{aznaran2024finite}
	{\sc F.~R. Aznaran, P.~E. Farrell, C.~W. Monroe, and A.~J. Van-Brunt}, {\em
		Finite element methods for multicomponent convection-diffusion}, IMA J.
	Numer. Anal.,  (2024), p.~drae001.
	
	\bibitem{zenodo}
	{\sc A.~Baier-Reinio and P.~E. Farrell}, {\em {Software used in `High-order
			finite element methods for three-dimensional multicomponent
			convection-diffusion'}}, 2025, 
			\url{https://doi.org/10.5281/zenodo.16416180}.
	
	\bibitem{petsc-user-ref}
	{\sc S.~Balay and {34 Others}}, {\em {PETSc/TAO} {U}sers {M}anual}, Tech.
	Report ANL-21/39 - Revision 3.21, Argonne National Laboratory, 2024,
	\url{https://doi.org/10.2172/2205494}.
	
	\bibitem{ngspetsc}
	{\sc J.~Betteridge, P.~E. Farrell, M.~Hochsteger, C.~Lackner, J.~Schöberl,
		S.~Zampini, and U.~Zerbinati}, {\em {ngsPETSc}: A coupling between
		{NETGEN}/{NGSolve} and {PETSc}}, J. Open Source Softw., 9 (2024), p.~7359.
	
	\bibitem{bird2002transport}
	{\sc R.~B. Bird, W.~E. Stewart, and E.~N. Lightfoot}, {\em Transport
		{P}henomena}, John Wiley \& Sons, 2nd~ed., 2002.
	
	\bibitem{boffi1997three}
	{\sc D.~Boffi}, {\em Three-dimensional finite element methods for the {S}tokes
		problem}, SIAM J. Numer. Anal., 34 (1997), pp.~664--670.
	
	\bibitem{boffi2013mixed}
	{\sc D.~Boffi, F.~Brezzi, and M.~Fortin}, {\em Mixed {F}inite {E}lement
		{M}ethods and {A}pplications}, Springer, Heidelberg, 2013.
	
	\bibitem{braukhoff2022entropy}
	{\sc M.~Braukhoff, I.~Perugia, and P.~Stocker}, {\em An entropy structure
		preserving space-time formulation for cross-diffusion systems: analysis and
		{G}alerkin discretization}, SIAM J. Numer. Anal., 60 (2022), pp.~364--395.
	
	\bibitem{brezzi1985two}
	{\sc F.~Brezzi, J.~Douglas, and L.~D. Marini}, {\em Two families of mixed
		finite elements for second order elliptic problems}, Numer. Math., 47 
		(1985),
	pp.~217--235.
	
	\bibitem{bulivcek2022existence}
	{\sc M.~Bul{\'\i}{\v{c}}ek, A.~J{\"u}ngel, M.~Pokorn{\`y}, and N.~Zamponi},
	{\em Existence analysis of a stationary compressible fluid model for
		heat-conducting and chemically reacting mixtures}, J. Math. Phys., 63 
		(2022).
	
	\bibitem{burman2003bunsen}
	{\sc E.~Burman, A.~Ern, and V.~Giovangigli}, {\em Bunsen flame simulation by
		finite elements on adaptively refined, unstructured triangulations}, 
		Combust.
	Theory Model., 8 (2003), p.~65.
	
	\bibitem{cances2023finite}
	{\sc C.~Cancès, V.~Ehrlacher, and L.~Monasse}, {\em Finite volumes for the
		{S}tefan–{M}axwell cross-diffusion system}, IMA J. Numer. Anal., 44 (2023),
	pp.~1029--1060.
	
	\bibitem{carnes2008local}
	{\sc B.~Carnes and G.~F. Carey}, {\em Local boundary value problems for the
		error in {FE} approximation of non-linear diffusion systems}, Internat. J.
	Numer. Methods Engrg., 73 (2008), pp.~665--684.
	
	\bibitem{chang1975some}
	{\sc H.-K. Chang, R.~C. Tai, and L.~E. Farhi}, {\em Some implications of
		ternary diffusion in the lung}, Resp. Physiol., 23 (1975), pp.~109--120.
	
	\bibitem{dalcinpazklercosimo2011}
	{\sc L.~D. Dalcin, R.~R. Paz, P.~A. Kler, and A.~Cosimo}, {\em Parallel
		distributed computing using {P}ython}, Adv. Water Resour., 34 (2011),
	pp.~1124--1139.
	
	\bibitem{deuflhard2011newton}
	{\sc P.~Deuflhard}, {\em Newton {M}ethods for {N}onlinear {P}roblems}, Springer
	Science \& Business Media, Berlin, Heidelberg, 2011.
	
	\bibitem{druet2021global}
	{\sc P.-{\'E}. Druet}, {\em Global--in--time existence for liquid mixtures
		subject to a generalised incompressibility constraint}, J. Math. Anal. 
		Appl.,
	499 (2021), p.~125059.
	
	\bibitem{ern1994multicomponent}
	{\sc A.~Ern and V.~Giovangigli}, {\em Multicomponent {T}ransport {A}lgorithms},
	vol.~24, Springer Berlin, Heidelberg, 1994.
	
	\bibitem{ern1998thermal}
	{\sc A.~Ern and V.~Giovangigli}, {\em Thermal diffusion effects in hydrogen-air
		and methane-air flames}, Combust. Theory Model., 2 (1998), p.~349.
	
	\bibitem{ern2021finiteI}
	{\sc A.~Ern and J.-L. Guermond}, {\em Finite {E}lements {I}: {A}pproximation
		and {I}nterpolation}, Springer, Cham, Switzerland, 2021.
	
	\bibitem{ern2021finiteII}
	{\sc A.~Ern and J.-L. Guermond}, {\em Finite {E}lements {II}: {G}alerkin
		{A}pproximation, {E}lliptic and {M}ixed {PDE}s}, Springer, Cham, 
		Switzerland,
	2021.
	
	\bibitem{ern2012mathematical}
	{\sc A.~Ern, R.~Joubaud, and T.~Leli{\`e}vre}, {\em Mathematical study of
		non-ideal electrostatic correlations in equilibrium electrolytes},
	Nonlinearity, 25 (2012), p.~1635.
	
	\bibitem{fick1855ueber}
	{\sc A.~Fick}, {\em {\"U}ber {D}iffusion}, Annalen der Physik, 170 (1855),
	pp.~59--86.
	
	\bibitem{giovangigli2012multicomponent}
	{\sc V.~Giovangigli}, {\em Multicomponent {F}low {M}odeling}, Birkhäuser,
	Boston, 1999.
	
	\bibitem{giovangigli2015steady}
	{\sc V.~Giovangigli, M.~Pokorn{\`y}, and E.~Zatorska}, {\em On the steady flow
		of reactive gaseous mixture}, Analysis (Berlin), 35 (2015), pp.~319--341.
	
	\bibitem{girault1986finite}
	{\sc V.~Girault and P.-A. Raviart}, {\em Finite {E}lement {M}ethods for
		{N}avier--{S}tokes {E}quations: {T}heory and {A}lgorithms}, 
		Springer-Verlag,
	Berlin, 1986.
	
	\bibitem{goyal2017new}
	{\sc P.~Goyal and C.~W. Monroe}, {\em New foundations of {N}ewman’s theory
		for solid electrolytes: thermodynamics and transient balances}, J.
	Electrochem. Soc., 164 (2017), pp.~E3647--E3660.
	
	\bibitem{doble2007perry}
	{\sc D.~W. Green and R.~H. Perry}, {\em Perry’s {C}hemical {E}ngineers’
		{H}andbook}, McGraw Hill Professional, 8th~ed., 2007.
	
	\bibitem{guggenheim1967thermodynamics}
	{\sc E.~A. Guggenheim}, {\em Thermodynamics: {A}n {A}dvanced {T}reatment for
		{C}hemists and {P}hysicists}, North-Holland Books, Amsterdam, 5th~ed., 
		1967.
	
	\bibitem{hager1989updating}
	{\sc W.~W. Hager}, {\em Updating the inverse of a matrix}, SIAM Rev., 31
	(1989), pp.~221--239.
	
	\bibitem{FiredrakeUserManual}
	{\sc D.~A. Ham and {26 Others}}, {\em Firedrake {U}ser {M}anual}, 1st~ed.,
	2023, \url{https://doi.org/10.25561/104839}.
	
	\bibitem{helfand1960inversion}
	{\sc E.~Helfand}, {\em On inversion of the linear laws of irreversible
		thermodynamics}, J. Chem. Phys., 33 (1960), pp.~319--322.
	
	\bibitem{john2016finite}
	{\sc V.~John}, {\em Finite {E}lement {M}ethods for {I}ncompressible {F}low
		{P}roblems}, Springer, Cham, Switzerland, 2016.
	
	\bibitem{john2017divergence}
	{\sc V.~John, A.~Linke, C.~Merdon, M.~Neilan, and L.~G. Rebholz}, {\em On the
		divergence constraint in mixed finite element methods for incompressible
		flows}, SIAM Rev., 59 (2017), pp.~492--544.
	
	\bibitem{joubaud2014numerical}
	{\sc R.~Joubaud, O.~Bernard, A.~Delville, A.~Ern, B.~Rotenberg, and P.~Turq},
	{\em Numerical study of density functional theory with mean spherical
		approximation for ionic condensation in highly charged confined
		electrolytes}, Phys. Rev. E, 89 (2014), p.~062302.
	
	\bibitem{jungel2015boundedness}
	{\sc A.~J{\"u}ngel}, {\em The boundedness-by-entropy method for cross-diffusion
		systems}, Nonlinearity, 28 (2015), p.~1963.
	
	\bibitem{jungel2019convergence}
	{\sc A.~J{\"u}ngel and O.~Leingang}, {\em Convergence of an implicit {E}uler
		{G}alerkin scheme for {P}oisson--{M}axwell--{S}tefan systems}, Adv. Comput.
	Math., 45 (2019), pp.~1469--1498.
	
	\bibitem{kraaijeveld1993negative}
	{\sc G.~Kraaijeveld and J.~A. Wesselingh}, {\em Negative {M}axwell--{S}tefan
		diffusion coefficients}, Ind. Eng. Chem. Res., 32 (1993), pp.~738--742.
	
	\bibitem{krishna2019diffusing}
	{\sc R.~Krishna}, {\em Diffusing uphill with {J}ames {C}lerk {M}axwell and
		{J}osef {S}tefan}, Chem. Eng. Sci., 195 (2019), pp.~851--880.
	
	\bibitem{krishna1997maxwell}
	{\sc R.~Krishna and J.~A. Wesselingh}, {\em The {M}axwell--{S}tefan approach to
		mass transfer}, Chem. Eng. Sci., 52 (1997), pp.~861--911.
	
	\bibitem{lightfoot1962applicability}
	{\sc E.~Lightfoot, E.~Cussler~Jr, and R.~Rettig}, {\em Applicability of the
		{Stefan--Maxwell} equations to multicomponent diffusion in liquids}, AIChE
	J., 8 (1962), pp.~708--710.
	
	\bibitem{longo2012finite}
	{\sc A.~Longo, M.~Barsanti, A.~Cassioli, and P.~Papale}, {\em A finite element
		{G}alerkin/least-squares method for computation of multicomponent
		compressible--incompressible flows}, Comput. \& Fluids, 67 (2012),
	pp.~57--71.
	
	\bibitem{maxwell1866}
	{\sc J.~C. Maxwell}, {\em On the dynamical theory of gases}, Phil. Trans. R.
	Soc.,  (1866), pp.~49--88.
	
	\bibitem{mcleod2014mixed}
	{\sc M.~McLeod and Y.~Bourgault}, {\em Mixed finite element methods for
		addressing multi-species diffusion using the {M}axwell--{S}tefan 
		equations},
	Comput. Meth. Appl. Mech. Eng., 279 (2014), pp.~515--535.
	
	\bibitem{nedelec1980mixed}
	{\sc J.-C. N{\'e}d{\'e}lec}, {\em Mixed finite elements in {$\mathbb{R}^3$}},
	Numer. Math., 35 (1980), pp.~315--341.
	
	\bibitem{nedelec1986new}
	{\sc J.-C. N{\'e}d{\'e}lec}, {\em A new family of mixed finite elements in
		{$\mathbb{R}^3$}}, Numer. Math., 50 (1986), pp.~57--81.
	
	\bibitem{newman2021electrochemical}
	{\sc J.~Newman and N.~P. Balsara}, {\em Electrochemical {S}ystems}, John Wiley
	\& Sons, Hoboken, NJ, 4th~ed., 2021.
	
	\bibitem{newman1965mass}
	{\sc J.~Newman, D.~Bennion, and C.~W. Tobias}, {\em Mass transfer in
		concentrated binary electrolytes}, Ber. Bunsenges. Phys. Chem., 69 (1965),
	pp.~608--612.
	
	\bibitem{onsager1931reciprocal}
	{\sc L.~Onsager}, {\em Reciprocal relations in irreversible processes. {I}.},
	Phys. Rev., 37 (1931), pp.~405--426.
	
	\bibitem{onsager1931reciprocal2}
	{\sc L.~Onsager}, {\em Reciprocal relations in irreversible processes. {II}.},
	Phys. Rev., 38 (1931), pp.~2265--2279.
	
	\bibitem{onsager1945theories}
	{\sc L.~Onsager}, {\em Theories and problems of liquid diffusion}, Ann. N. Y.
	Acad. Sci., 46 (1945), pp.~241--265.
	
	\bibitem{raviart1977mixed}
	{\sc P.-A. Raviart and J.-M. Thomas}, {\em A mixed finite element method for
		2-nd order elliptic problems}, in {M}athematical {A}spects of {F}inite
	{E}lement {M}ethods, vol.~606 of {L}ecture {N}otes in {M}ath., Springer,
	Berlin, 1977, pp.~{292--315}.
	
	\bibitem{schoberl1997netgen}
	{\sc J.~Sch\"oberl}, {\em {NETGEN}: An advancing front {2D}/{3D}-mesh generator
		based on abstract rules}, Computing and Visualization in Science, 1 (1997),
	pp.~41--52.
	
	\bibitem{scott1985norm}
	{\sc L.~R. Scott and M.~Vogelius}, {\em Norm estimates for a maximal right
		inverse of the divergence operator in spaces of piecewise polynomials}, 
		ESAIM
	Math. Model. Numer. Anal., 19 (1985), pp.~111--143.
	
	\bibitem{stefan1871gleichgewicht}
	{\sc J.~Stefan}, {\em {\"U}ber das {G}leichgewicht und die {B}ewegung,
		insbesondere die {D}iffusion von {G}asgemengen}, Sitzber. Akad. Wiss. 
		Wien.,
	63 (1871), pp.~63--124.
	
	\bibitem{sun2019entropy}
	{\sc Z.~Sun, J.~A. Carrillo, and C.-W. Shu}, {\em An entropy stable high-order
		discontinuous {G}alerkin method for cross-diffusion gradient flow systems},
	Kinet. Relat. Models, 12 (2019), pp.~885--908.
	
	\bibitem{taylor1973numerical}
	{\sc C.~Taylor and P.~Hood}, {\em A numerical solution of the
		{N}avier--{S}tokes equations using the finite element technique}, Comput. 
		\&
	Fluids, 1 (1973), pp.~73--100.
	
	\bibitem{thorstenson1989gas}
	{\sc D.~C. Thorstenson and D.~W. Pollock}, {\em Gas transport in unsaturated
		zones: {M}ulticomponent systems and the adequacy of {F}ick's laws}, Water
	Resour. Res., 25 (1989), pp.~477--507.
	
	\bibitem{van2022augmented}
	{\sc A.~Van-Brunt, P.~E. Farrell, and C.~W. Monroe}, {\em Augmented
		saddle-point formulation of the steady-state {S}tefan--{M}axwell diffusion
		problem}, IMA J. Numer. Anal., 42 (2022), pp.~3272--3305.
	
	\bibitem{wankat2022separation}
	{\sc P.~C. Wankat}, {\em Separation {P}rocess {E}ngineering: {I}ncludes {M}ass
		{T}ransfer {A}nalysis}, Pearson, 5th~ed., 2022.
	
	\bibitem{wesselingh2000mass}
	{\sc J.~Wesselingh and R.~Krishna}, {\em Mass {T}ransfer in {M}ulticomponent
		{M}ixtures}, Delft University Press, Delft, Netherlands, 2000.
	
	\bibitem{zhang2004new}
	{\sc S.~Zhang}, {\em A new family of stable mixed finite elements for the {3D
			S}tokes equations}, Math. Comp., 74 (2005), pp.~543--555.
	
\end{thebibliography}

\end{document}